\documentclass{amsart}
\usepackage{amsthm,amsmath,amssymb,amscd, mathrsfs}
\usepackage[latin1]{inputenc}
\usepackage[all]{xy}
\usepackage{url}
\numberwithin{equation}{section}

\theoremstyle{plain}
\newtheorem{thm}{Theorem}[section]
\newtheorem*{thm*}{Theorem}
\newtheorem{prop}[thm]{Proposition}
\newtheorem{lem}[thm]{Lemma}
\newtheorem{cor}[thm]{Corollary}

\newtheorem{lem-defn}[thm]{Lemma-Definition}
\newtheorem{claim}[thm]{Claim}
\newtheorem*{claim*}{Claim}

\theoremstyle{definition}
\newtheorem{defn}[thm]{Definition}
\newtheorem{example}[thm]{Example}

\theoremstyle{remark}
\newtheorem{rem}[thm]{Remark}

\DeclareMathOperator{\val}{val}
\DeclareMathOperator{\Frac}{Frac}
\DeclareMathOperator{\Tot}{Tot}
\DeclareMathOperator{\Ker}{Ker}

\DeclareMathOperator{\Image}{Im}
\DeclareMathOperator{\Hom}{Hom}
\DeclareMathOperator{\End}{End}

\DeclareMathOperator{\GL}{GL}
\DeclareMathOperator{\id}{id}

\DeclareMathOperator{\tr}{tr}
\DeclareMathOperator{\rank}{rank}
\DeclareMathOperator{\Fil}{Fil}
\DeclareMathOperator{\gr}{gr}

\DeclareMathOperator{\Gal}{Gal}

\DeclareMathOperator{\Spec}{Spec}
\DeclareMathOperator{\Spa}{Spa}

\newcommand{\ra}{\rightarrow}
\newcommand{\lra}{\longrightarrow}
\newcommand{\lmt}{\longmapsto}
\newcommand{\hra}{\hookrightarrow}

\newcommand{\geom}{\mathrm{geom}}

\newcommand{\dR}{\mathrm{dR}}
\newcommand{\HT}{\mathrm{HT}}

\newcommand{\Higgs}{\mathrm{Higgs}}
\newcommand{\et}{\mathrm{\acute{e}t}}
\newcommand{\proet}{\mathrm{pro\acute{e}t}}
\newcommand{\fin}{\mathrm{fin}}
\newcommand{\hatO}{\hat{\mathcal{O}}}
\newcommand{\OC}{\mathcal{O}\mathbb{C}}
\newcommand{\OB}{\mathcal{O}\mathbb{B}}
\newcommand{\RH}{\mathcal{RH}}
\newcommand{\Res}{\mathrm{Res}}

\newcommand{\N}{\mathbb{N}}
\newcommand{\Z}{\mathbb{Z}}
\newcommand{\Q}{\mathbb{Q}}

\newcommand{\C}{\mathbb{C}}
\newcommand{\A}{\mathbb{A}}

\newcommand{\fkp}{\mathfrak{p}}

\newcommand{\fkP}{{\mathfrak P}}

\newcommand{\bB}{\mathbb{B}}

\newcommand{\bL}{\mathbb{L}}

\newcommand{\bT}{\mathbb{T}}

\newcommand{\calA}{\mathcal{A}}
\newcommand{\calB}{\mathcal{B}}

\newcommand{\calE}{\mathcal{E}}
\newcommand{\calF}{\mathcal{F}}
\newcommand{\calG}{\mathcal{G}}
\newcommand{\calH}{\mathcal{H}}

\newcommand{\calM}{\mathcal{M}}

\newcommand{\calO}{\mathcal{O}}

\newcommand{\calU}{\mathcal{U}}

\newcommand{\calX}{\mathcal{X}}

\begin{document}
\title{
Constancy of generalized Hodge-Tate weights of a local system
}
\date{\today}
\author{Koji Shimizu}
\address{Harvard University, Department of Mathematics, Cambridge, MA 02138}
\email{shimizu@math.harvard.edu}

\begin{abstract}
Sen attached to each $p$-adic Galois representation of a $p$-adic field
a multiset of numbers called generalized Hodge-Tate weights.
In this paper, we discuss a rigidity of these numbers in a geometric family.
More precisely, we consider a $p$-adic local system on a rigid analytic variety over a $p$-adic field
and show that the multiset of generalized Hodge-Tate weights of the local system is constant.
The proof uses the $p$-adic Riemann-Hilbert correspondence by Liu and Zhu, a Sen-Fontaine decompletion theory in the relative setting, and the theory of formal connections.
We also discuss basic properties of Hodge-Tate sheaves on a rigid analytic variety.
\end{abstract}

\maketitle

\section{Introduction}

In the celebrated paper \cite{Tate}, Tate studied the Galois cohomology of $p$-adic fields and obtained the so-called Hodge-Tate decomposition
of the Tate module of a $p$-divisible group with good reduction.
The paper has been influential in the developments of $p$-adic Hodge theory,
and one of the earliest progresses was done by Sen.
In \cite{Sen-cont}, he attached to each $p$-adic Galois representation of a $p$-adic field $k$ a multiset of numbers that are algebraic over $k$. These numbers are called generalized Hodge-Tate weights, and
they serve as one of the basic invariants in $p$-adic Hodge theory,
especially for the study of Galois representations that may not be Hodge-Tate (e.g.~Galois representations attached to finite slope overconvergent modular forms).

In this paper, we study how generalized Hodge-Tate weights vary in a geometric family. 
To be precise, we consider an \'etale $\Q_p$-local system on a rigid analytic varieties over $k$ and regard it as a family of Galois representations of residue fields of its classical points.
Here is one of the main theorems of this paper.

\begin{thm}[Corollary~\ref{cor:constancy of generalized HT weights}]
\label{intro:constancy}
Let $X$ be a geometrically connected smooth rigid  analytic variety over $k$
and let $\bL$ be a $\Q_p$-local system on $X$.
Then the generalized Hodge-Tate weights
of the $p$-adic Galois representations $\bL_{\overline{x}}$ of  $k(x)$
are constant on the set of classical points $x$ of $X$.
\end{thm}

The theorem gives one instance of the rigidity of a geometric family of Galois representations.
It is worth noting that arithmetic families of Galois representations do not have such rigidity; 
consider a representation of the absolute Galois group of $k$
 with coefficients in some $\Q_p$-affinoid algebra. One can associate to each maximal ideal a Galois representation of $k$.
In such a situation, the generalized Hodge-Tate weights vary over the maximal ideals.

To explain ideas of the proof of Theorem~\ref{intro:constancy}
as well as other results of this paper, 
let us recall the work of Sen mentioned above.
For each $p$-adic Galois representation $V$ of $k$, 
we set 
\[
 \calH(V):=(V\otimes_{\Q_p}\C_p)^{\Gal(\overline{k}/k_\infty)},
\]
where $\C_p$ is the $p$-adic completion of $\overline{k}$
and $k_\infty:=k(\mu_{p^\infty})$ is the cyclotomic extension of $k$. 
This is a vector space over the $p$-adic completion $K$
of $k_\infty$ equipped with a continuous semilinear action of $\Gal(k_\infty/k)$ and satisfies $\dim_K \calH(V)=\dim_{\Q_p}V$.
Sen developed a theory of decompletion;
he found a natural $k_\infty$-vector subspace 
$\calH(V)_{\fin}\subset \calH(V)$ 
that is stable under $\Gal(k_\infty/k)$-action
and satisfies $\calH(V)_{\fin}\otimes_{k_\infty}K = \calH(V)$.
He then defined a $k_\infty$-endomorphism $\phi_V$ on $\calH(V)_{\fin}$,
called the Sen endomorphism of $V$,
by considering the infinitesimal action of $\Gal(k_\infty/k)$.
The generalized Hodge-Tate weights are defined to be eigenvalues of $\phi_V$.

Therefore, the first step toward Theorem~\ref{intro:constancy}
is to define generalizations of $\calH(V)$ and $\phi_V$
for each $\Q_p$-local system.
For this, we use the $p$-adic Simpson correspondence by Liu and Zhu
\cite{LZ};
based on recent developments in relative $p$-adic Hodge theory
by Kedlaya-Liu and Scholze, 
Liu and Zhu associated to each $\Q_p$-local system $\bL$ on $X$
a vector bundle $\calH(\bL)$ of the same rank on $X_K$
equipped with a $\Gal(k_\infty/k)$-action
and a Higgs field, where $X_K$ is the base change of $X$ to $K$.
When $X$ is a point and $\bL$ corresponds to $V$, this agrees 
with $\calH(V)$ as the notation suggests.
Following Sen, we will define 
the arithmetic Sen endomorphism $\phi_{\bL}$ of $\bL$
by decompleting $\calH(\bL)$ and considering
the infinitesimal action of $\Gal(k_\infty/k)$.
Then Theorem~\ref{intro:constancy} is reduced to the following:

\begin{thm}[Theorem~\ref{thm:constancy of eigenvalues of arithmetic Sen operator}]\label{intro:constancy of Sen endom}
The eigenvalues of $\phi_{\bL,x}$ for $x\in X_K$ are algebraic over $k$ 
and constant on $X_K$.
\end{thm}

Before discussing ideas of the proof,
let us mention consequences of Theorem~\ref{intro:constancy of Sen endom}.
Sen proved that a $p$-adic Galois representation $V$ is Hodge-Tate
if and only if $\phi_V$ is semisimple with integer eigenvalues.
In the same way, we use $\phi_{\bL}$ to study Hodge-Tate sheaves.
We define a sheaf $D_{\HT}(\bL)$ on the \'etale site $X_{\et}$
by
\[
 D_{\HT}(\bL):=\nu_\ast(\bL\otimes_{\Q_p}\OB_{\HT}),
\]
where $\OB_{\HT}$ is the Hodge-Tate period sheaf
on the pro-\'etale site $X_{\proet}$ and $\nu\colon X_{\proet}\ra X_{\et}$
is the projection (see Section~\ref{section:applications}).
A $\Q_p$-local system $\bL$ is called \emph{Hodge-Tate}
if $D_{\HT}(\bL)$ is a vector bundle on $X$ of rank equal to $\rank\bL$.

\begin{thm}[Theorem~\ref{thm:def of Hodge-Tate sheaves}]
\label{thm:Intro defo of Hodge-Tate sheaves}
 The following conditions are equivalent for a $\Q_p$-local system $\bL$ on $X$:
\begin{enumerate}
 \item $\bL$ is Hodge-Tate.
 \item $\phi_{\bL}$ is semisimple with integer eigenvalues.
\end{enumerate}
\end{thm}

The study of the Sen endomorphism for a geometric family
was initiated by Brinon as a generalization of Sen's theory
to the case of non-perfect residue fields (\cite{Brinon}).
Tsuji obtained Theorem~\ref{thm:Intro defo of Hodge-Tate sheaves} in the case of schemes with semistable reduction (\cite{Tsuji}).

Using this characterization, we prove the following basic property
of Hodge-Tate sheaves:

\begin{thm}[Theorem~\ref{thm:pushforward of Hodge-Tate sheaves}]
\label{intro:pushforward}
Let $f \colon X\ra Y$ be a smooth proper morphism 
between smooth rigid analytic varieties over $k$
and let $\bL$ be a $\Z_p$-local system on $X_{\et}$.
Assume that $R^if_\ast \bL$ is a $\Z_p$-local system on $Y_{\et}$ for each $i$.
Then if $\bL$ is a Hodge-Tate sheaf on $X_{\et}$, 
$R^if_\ast \bL$ is a Hodge-Tate sheaf on $Y_{\et}$.
\end{thm}

Hyodo introduced the notion of Hodge-Tate sheaves
and proved Theorem~\ref{intro:pushforward}
in the case of schemes (\cite{Hyodo}).
Links between
Hodge-Tate sheaves and the $p$-adic Simpson correspondence
can be seen in his work and were also studied by Abbes-Gros-Tsuji
(\cite{AGT}) and Tsuji (\cite{Tsuji-Simpson}). In fact, they undertook a systematic development of the $p$-adic Simpson correspondence started by Faltings  \cite{Faltings-Simpson}
and their focus is much broader than ours.
Andreatta and Brinon also studied Higgs modules and Sen endomorphisms
in a different setting (\cite{AB}). 
In these works, one is restricted to working with schemes or log schemes,
whereas we work with rigid analytic varieties.

We now turn to the proof of Theorem~\ref{intro:constancy of Sen endom}.
The key idea to obtain such constancy
is to describe $\phi_{\bL}$ as the residue of a certain formal integrable connection.
Such an idea occurs in the work \cite{AB} of Andreatta and Brinon.
Roughly speaking, they associated to $\bL$ 
a formal connection over some pro-\'etale cover
of $X_K$ when $X$ is an affine scheme admitting invertible coordinates. 
In our case, we want to work over $X_K$, and thus
we use the geometric $p$-adic Riemann-Hilbert correspondence
by Liu and Zhu \cite{LZ} and Fontaine's decompletion theory for 
the de Rham period ring $B_{\dR}(K)$ in the relative setting.

Liu and Zhu associated to each $\Q_p$-local system $\bL$ on $X$
a locally free $\calO_X\hat{\otimes}B_{\dR}(K)$-module $\RH(\bL)$
equipped with an integrable connection 
\[
 \nabla\colon \RH(\bL)\ra \RH(\bL)\otimes \Omega^1_X
\]
and a $\Gal(k_\infty/k)$-action
(see Subsection~\ref{subsection:review of gemetric Riemann-Hilbert} for the 
notation).
To regard $\phi_{\bL}$ as a residue, we also need a connection in the arithmetic direction $B_{\dR}(K)$.
For this we use Fontaine's decompletion theory \cite{Fontaine};
recall the natural inclusion $k_\infty((t))\subset B_{\dR}(K)$
where $t$ is the $p$-adic analogue of the complex period $2\pi i$.
Fontaine extended the work of Sen and developed a decompletion theory for 
$B_{\dR}(K)$-representations of $\Gal(k_\infty/k)$.
We generalize Fontaine's decompletion theory to the relative setting, i.e., 
that for $\calO_X\hat{\otimes}B_{\dR}(K)$-modules (Theorem~\ref{thm:Fontaine decompletion} and Proposition~\ref{prop:Fontaine connection}),
which yields an endomorphism $\phi_{\dR,\bL}$ 
on $\RH(\bL)_{\fin}$
satisfying
\[
 \phi_{\dR,\bL}(t^n v)=nt^nv+t^n\phi_{\dR,\bL}(v)
\]
and $\gr^0 \phi_{\dR,\bL}=\phi_{\bL}$.
Informally, this means that we have
an integrable connection 
\[
 \nabla+ \frac{\phi_{\dR,\bL}}{t}\otimes dt
\colon \RH(\bL)\ra \RH(\bL)\otimes 
\bigl((\calO_X\hat{\otimes}B_{\dR}(K))\otimes\Omega^1_X 
+(\calO_X\hat{\otimes}B_{\dR}(K))\otimes dt \bigr)
\]
over $X\hat{\otimes}B_{\dR}(K)$
whose residue along $t=0$ coincides with 
the arithmetic Sen endomorphism $\phi_{\bL}$.
We develop a theory of formal connections to analyze our connection
and prove Theorem~\ref{intro:constancy of Sen endom}.

Finally, let us mention two more results in this paper.
The first result is a rigidity of Hodge-Tate local systems
of rank at most two.

\begin{thm}[Theorem~\ref{thm:principle B for rank 2}]
Let $X$ be a geometrically connected smooth rigid analytic variety over $k$
and let $\bL$ be a $\Q_p$-local system on $X_{\et}$.
Assume that $\rank \bL$ is at most two.
If $\bL_{\overline{x}}$ is a Hodge-Tate representation at a classical point $x\in X$, then $\bL$ is a Hodge-Tate sheaf.
In particular, $\bL_{\overline{y}}$ is a Hodge-Tate representation at every classical point $y\in X$
\end{thm}
\noindent
Liu and Zhu proved such a rigidity for de Rham local systems
(\cite[Theorem~1.3]{LZ}).
We do not know whether a similar statement holds for
Hodge-Tate local systems of higher rank.

The second result concerns the relative $p$-adic monodromy conjecture
for de Rham local systems; the conjecture states that
a de Rham local system on $X$
becomes semistable at every classical point
after a finite \'etale extension of $X$
(cf.~\cite[\S 0.8]{KL-I}, \cite[Remark 1.4]{LZ}).
This is a relative version of the $p$-adic monodromy theorem proved by Berger
(\cite{Berger}), and it is a major open problem in relative $p$-adic Hodge theory.
We work on the case of de Rham local systems with a single Hodge-Tate weight,
in which case the result follows from a theorem of Sen (Theorem~\ref{thm:Sen}).
\begin{thm}[Theorem~\ref{thm:weight zero p-adic monodromy}]
\label{thm:intro p-adic monodromy}
Let $X$ be a smooth rigid analytic variety over $k$
and let $\bL$ be a $\Z_p$-local system on $X_{\et}$.
Assume that $\bL$ is a Hodge-Tate sheaf with a single Hodge-Tate weight.
Then there exists a finite \'etale cover $f\colon Y\ra X$ such that
$(f^\ast \bL)_{\overline{y}}$ is semistable at every classical point 
$y$ of $Y$.
 \end{thm}

This is the simplest case of the relative $p$-adic monodromy conjecture.
In \cite{Colmez}, Colmez gave a proof of the $p$-adic monodromy theorem 
for de Rham Galois representations
using Sen's theorem mentioned above.
It is an interesting question
whether one can adapt Colmez's strategy to
the relative setting using Theorem~\ref{thm:intro p-adic monodromy}.

The organization of the paper is as follows:
Section~\ref{section:decompletion}
presents Sen-Fontaine's decompletion theory in the relative setting.
In Section~\ref{section:arithmetic Sen endomorphism},
we review the $p$-adic Simpson correspondence by Liu and Zhu,
and define the arithmetic Sen endomorphism $\phi_{\bL}$.
Section~\ref{section:constancy}
discusses a Fontaine-type decompletion for 
the geometric $p$-adic Riemann-Hilbert correspondence by Liu and Zhu,
and develops a theory of formal connections.
Combining them together we prove Theorem~\ref{intro:constancy}.
Section~\ref{section:applications} presents
applications of the study of the arithmetic Sen endomorphism
including basic properties of Hodge-Tate sheaves, a rigidity of Hodge-Tate sheaves, and the relative $p$-adic monodromy conjecture.

\noindent {\bf Conventions.}
We will use Huber's adic spaces as our language for non-archimedean analytic geometry. In particular, a rigid analytic variety over $\Q_p$
will refer to a quasi-separated adic space that is locally of finite type over
$\Spa (\Q_p,\Z_p)$ (\cite[\S 4]{Huber-gen}, \cite[1.11.1]{Huber-etale}).

We will use Scholze's theory of perfectoid spaces and pro-\'etale site.
For the pro-\'etale site, we will use the one introduced in \cite{Scholze, Scholze-errata}.

\medskip

\noindent {\bf Acknowledgment.} 
The author thanks Mark Kisin for his constant encouragement and useful feedbacks.

\section{Sen-Fontaine's decompletion theory for an arithmetic family}
\label{section:decompletion}
\subsection{Set-up}\label{subsection:setup}

 Let $k$ be a finite field extension of $\Q_p$.
We set $k_m:=k(\mu_{p^m})$ and $k_\infty:=\varinjlim_m k_m$.
Let $K$ denote the $p$-adic completion of $k_{\infty}$.
We set $\Gamma_k:=\Gal(k_\infty/k)$. 
Then $\Gamma_k$ is identified with an open subgroup of $\Z_p^\times$
via the cyclotomic character $\chi\colon \Gamma_k\ra \Z_p^\times$
and it acts continuously on $K$.

Let $L_{\dR}^+$ (resp.~ $L_{\dR}$) denote the de Rham period ring $B_{\dR}^+(K)$
(resp.~ $B_{\dR}(K)$) introduced by Fontaine.
We fix a compatible sequence of $p$-power roots of unity $(\zeta_{p^n})$ and set $t:=\log [\varepsilon]$ where $\varepsilon=(1,\zeta_p,\zeta_{p^2},\ldots)\in \calO_{K^\flat}$.
Then $\Gamma_k$ acts on $t$ via the cyclotomic character and
the $\Z_p$-submodule $\Z_pt\subset L_{\dR}^+$ does not depend on the choice of $(\zeta_{p^n})$. Note that $L_{\dR}$ is a discrete valuation ring
with the residue field $K$, the fraction field $L_{\dR}$, and a uniformizer $t$, and that
$k_{\infty}[[t]]$ is embedded into $L_{\dR}^+$.

We now recall the Sen-Fontaine's decompletion theory 
(\cite[Theorem 3]{Sen-cont}, \cite[Th\'eor\`eme 3.6]{Fontaine}).

\begin{thm}
\hfill
\begin{enumerate}
 \item \emph{(Sen)}
Let $V$ be a $K$-representation of $\Gamma_k$.
Denote by $V_{\fin}$ the union of
finite-dimensional $k$-vector subspaces of $V$
that are stable under the action of $\Gamma_k$.
Then the natural map
\[
 V_{\fin}\otimes_{k_\infty}K\ra V
\]
is an isomorphism.
 \item \emph{(Fontaine)}
Let $V$ be an $L_{dR}^+$-representation of $\Gamma_k$
and set 
\[
 V_{\fin}:=\varprojlim_n (V/t^nV)_{\fin},
\]
where $(V/t^nV)_{\fin}$
is defined to be the union of 
finite-dimensional $k$-vector subspaces of $V/t^nV$
that are stable under the action of $\Gamma$.
Then the natural map
\[
 V_{\fin}\otimes_{k_\infty[[t]]}L_{\dR}^+
\]
is an isomorphism.
\end{enumerate} 
\end{thm}

Using this theorem, Sen defined the so-called Sen endomorphism $\phi_V$
on $V_\infty$ for a $K$-representation $V$ of $\Gamma_k$ (cf.~\cite[Theorem 4]{Sen-cont}), 
and Fontaine defined a formal connection on $V_{\fin}$
for a $L_{\dR}^+$-representation $V$ of $\Gamma_k$ 
(cf.~\cite[Proposition 3.7]{Fontaine}).

We now turn to the relative setting.
Let $A$ be a Tate $k$-algebra that is reduced and topologically of finite type over $k$.
It is equipped with the supremum norm and we use this norm when we regard $A$ as a Banach $k$-algebra.
We further assume that $(A,A^\circ)$ is smooth over $(k,\calO_k)$. 
We set
\[
 A_{k_m}:=A\hat{\otimes}_kk_m,\qquad
A_\infty:=\varinjlim_m A_{k_m},\qquad \text{and}\qquad 
A_K:=A\hat{\otimes}_k K.
\]

Here we use a slightly heavy notation $A_{k_m}$
to reserve $A_m$ for a different ring in a later section.
Since $A$, $k_m$ and $K$ are all complete Tate $k$-algebras, the completed tensor product is well-defined (or one can use Banach $k$-algebra structures).
Note that $A_{k_m}$ (resp. $A_K$) is a complete Tate $k_m$-algebra (resp. $K$-algebra),
that $A_\infty$ is a Tate $k_\infty$-algebra and
that $A_K$ is the completion of $A_\infty$.

We introduce the relative versions of $k_\infty[[t]]$, $L^+_{\dR}$ and $L_{\dR}$
over $A$.
We set
\[
 A_\infty[[t]]:=\varprojlim_n A_\infty[t]/(t^n),
\]
and equip $A_{\infty}[[t]]$ with the inverse limit topology 
of Tate $k_{\infty}$-algebras $A_\infty[t]/(t^n)$.
We also set
\[
 A\hat{\otimes}L_{\dR}^+:=\varprojlim_n A\hat{\otimes}_k L_{\dR}^+/(t^n),
\]
and equip $A\hat{\otimes}L_{\dR}^+$ with
the inverse limit topology.
We finally set
\[
 A\hat{\otimes}L_{\dR}=(A\hat{\otimes}L_{\dR}^+)[t^{-1}]
\]
and equip $A\hat{\otimes}L_{\dR}$ with the inductive limit topology.
Note that $\Gamma_k$ acts continuously on these rings
(cf.~\cite[Appendix]{Bellovin}).

\begin{defn}\label{defn:representation}
In this paper, an \textit{$A\hat{\otimes}L_{\dR}^+$-representation} of $\Gamma_k$ 
is an $A\hat{\otimes}L_{\dR}^+$-module $V$ 
that is isomorphic to either $(A\hat{\otimes}L_{\dR}^+)^r$
or $(A\hat{\otimes}L_{\dR}^+/(t^n))^r$ for some $r$ and $n$,
equipped with a continuous
$A\hat{\otimes}L_{\dR}^+$-semilinear action of $\Gamma_k$.
 We denote the category of 
$A\hat{\otimes}L_{\dR}^+$-representations of $\Gamma_k$ 
by $\mathrm{Rep}_{\Gamma_k}(A\hat{\otimes}L_{\dR}^+)$.
An $A\hat{\otimes}L_{\dR}^+$-representation of $\Gamma_k$ that is annihilated by $t$ is also called an \textit{$A_K$-representation of} $\Gamma_k$.
\end{defn}

If $V$ is isomorphic to either $(A\hat{\otimes}L_{\dR}^+)^r$
or $(A\hat{\otimes}L_{\dR}^+/(t^n))^r$ then
$V$ admits a topology by taking a basis and the topology is independent of the choice of the basis. Thus the continuity condition of the action of $\Gamma_k$ makes sense.
Note that if $V$ is an $A\hat{\otimes}L_{\dR}^+$-representation of $\Gamma_k$, then so are $t^nV$ and $V/t^nV$.

We are going to discuss the relative version of Sen-Fontaine's theory.
Namely, we will work on $A_K$-representations of $\Gamma_k$
and $A\hat{\otimes}L_{\dR}^+$-representations of $\Gamma_k$.
Note that Sen's theory in the relative setting is established by Sen himself
(\cite{Sen-variation, Sen-infinite-dim})
and that Fontaine's decompletion theory in the relative setting
is established by Berger-Colmez and Bellovin for representations which 
come from $A$-representations of $\Gal(\overline{k}/k)$
via the theory of $(\varphi,\Gamma)$-modules (\cite{Berger-Colmez, Bellovin}).
Since we need a Fontaine-type decompletion theory for arbitrary
$A\hat{\otimes}L_{\dR}^+$-representations of $\Gamma_k$, 
we give detailed arguments;
we will discuss the decompletion theory in the next subsection,
and define Sen's endomorphism and Fontaine's connection in 
Subsection \ref{subsection:Sen's operator and Fontaine's connection}.

We end this subsection with establishing basic properties of 
the rings we have introduced.

\begin{prop}\hfill
\begin{enumerate}
 \item 
For each $m\geq 1$, $A\hat{\otimes}_k L_{\dR}^+/(t^n)$ is Noetherian
and faithfully flat over $A_{\infty}[t]/(t^n)$.
 \item $A\hat{\otimes}L_{\dR}^+$ is a $t$-adically complete flat $L_{\dR}^+$-algebra
with $(A\hat{\otimes}L_{\dR}^+)/(t^n)=A\hat{\otimes}_k L_{\dR}^+/(t^n)$.
\end{enumerate}
 \end{prop}

\begin{proof}
For (i), the first assertion is proved in \cite[Lemma 13.3]{BMS}.
We prove that $A\hat{\otimes}_k L_{\dR}^+/(t^n)$ is faithfully flat
over $A_{\infty}[t]/(t^n)$.

First we deal with the case $n=1$, i.e., faithful flatness of 
$A_K$ over $A_\infty$. The proof is similar to that of \cite[Lemme 5.9]{AB}.
Recall $A_\infty=\varinjlim_n A_{k_m}$.
Since $k_m$ and $K$ are both complete valuation fields,
$A_K=A_{k_m}\hat{\otimes}_{k_m}K$ is faithfully flat over $A_{k_m}$
(e.g. use \cite[Proposition 2.1.7/8 and Theorem 2.8.2/2]{BGR}).

We prove that $A_K$ is flat over $A_\infty$. 
For this it suffices to show that for any finitely generated ideal $I\subset A_\infty$, the map $I\otimes_{A_\infty}A_K \ra A_K$ is injective.
Take such an ideal $I$. As $I$ is finitely generated, there exist a positive integer $m$ and a finitely generated ideal $I_m\subset A_{k_m}$
such that $I=\Image (I_m\otimes_{A_{k_m}}A_\infty \ra A_\infty )$.
Since $A_K$ is flat over $A_{k_m}$, the map $I_m\otimes_{A_{k_m}}A_K \ra A_K$
is injective.
On the other hand, this map factors as 
$I_m\otimes_{A_{k_m}}A_K \ra I\otimes_{A_\infty}A_K \ra A_K$
and the first map is surjective by the choice of $I_m$.
Hence the second map $I\otimes_{A_\infty}A_K\ra A_K$ is injective.

For faithful flatness, 
it remains to prove that the map
$\Spec A_K \ra \Spec A_\infty$ is surjective.
Assume the contrary and take a prime ideal $\fkP\in \Spec A_\infty$
that is not in the image of the map.
Set $\fkp=\fkP\cap A \in \Spec A$. 
Note that the prime ideals of $A_\infty$ above $\fkp$ are 
conjugate to each other by the action of $\Gamma_k$.
From this we see that no prime ideal of $A_\infty$ above $\fkp$
is in the image of  $\Spec A_K \ra \Spec A_\infty$.
Hence $\fkp$ does not lie in the image of $\Spec A_K \ra \Spec A$, 
which contradicts that $A_K$ is faithfully flat over $A$.

Next we deal with the general $n$.
By the local flatness criterion (\cite[Theorem 22.3]{Matsumura})
applied to the nilpotent ideal $(t)\subset A_{\infty}[t]/(t^n)$,
the flatness follows from the case $n=1$.
Moreover, since $\Spec A_K \ra \Spec A_\infty$ is surjective,
so is $\Spec A\hat{\otimes}_k L_{\dR}^+/(t^n) \ra \Spec A_\infty[t]/(t^n)$.
Hence $A\hat{\otimes}_k L_{\dR}^+/(t^n)$ is faithfully flat
over $A_{\infty}[t]/(t^n)$.

Assertion (ii) is proved in \cite[Lemma 13.4]{BMS}.
Note that the proof of \cite[Lemma 13.4]{BMS} works in our setting 
since we assume the smoothness of $A$.
\end{proof}

\subsection{Sen-Fontaine's decompletion theory in the relative setting}

\begin{defn}
For an $A\hat{\otimes}L_{\dR}^+$-representation $V$ of $\Gamma_k$,
we define the subspace $V_{\fin}$ as follows:
\begin{itemize}
 \item If $V$ is annihilated by $t^n$ for some $n\geq 1$, then $V_{\fin}$ is defined to be the union of finitely generated $A$-submodules of $V$ that are stable under the action of $\Gamma_k$.
 \item In general, define 
\[
 V_{\fin} := \varprojlim_n (V/t^nV)_{\fin}.
\]
\end{itemize}
\end{defn}

If $V$ is killed by $t^n$, then $V_{\fin}$ is an $A_{\infty}[t]/(t^n)$-module. In general, $V_{\fin}$ is an $A_{\infty}[[t]]$-module equipped with a semilinear action of $\Gamma_k$.

The following theorem is the main goal of this subsection.

\begin{thm}\label{thm:Fontaine decompletion}
For an $A\hat{\otimes}L_{\dR}^+$-representation $V$  of $\Gamma_k$
that is finite free of rank $r$ over $A\hat{\otimes}L_{\dR}^+$,
the $A_\infty[[t]]$-module $V_{\fin}$
is finite free of rank $r$.
Moreover, the natural map
\[
 V_{\fin}\otimes_{A_{\infty}[[t]]}(A\hat{\otimes}L_{\dR}^+)
\ra V
\]
is an isomorphism, and $V_{\fin}/t^nV_{\fin}$
is isomorphic to $(V/t^nV)_{\fin}$ for each $n\geq 1$.
\end{thm}

The key tool in the proof is the Sen method,
which is axiomatized in 
\cite[\S 3]{Berger-Colmez}.
We review parts of the Tate-Sen conditions 
that are used in our proofs.
For a thorough treatment, we refer the reader to \textit{loc.~cit.}

Consider Tate's normalized trace map
\[
 R_{k,m}=R_m\colon K\ra k_m.
\]
On $k_{m+m'}\subset K$, this map is defined as
\[
 [k_{m+m'}:k_m]^{-1}\tr_{k_{m+m'}/k_m}\colon k_{m+m'}\ra k_m,
\]
and it extends continuously to $R_{k,m}\colon K\ra k_m$.
We denote the kernel $\Ker R_{k,m}$ by $X_m$.
The map $R_{k,m}$ extends $A$-linearly to 
the map $R_{A,m}\colon A_K\ra A_{k_m}$.
Fix a real number $c_3>1$.
By work of Tate and Sen 
(\cite[Proposition 3.1.4 and Proposition 4.1.1]{Berger-Colmez}),
$G_0=\Gamma_k$, $\tilde{\Lambda}=A_K$, $R_m$, and the valuation $\val$ on $A_K$
satisfy the Tate-Sen axioms in \cite[\S 3]{Berger-Colmez}
for any fixed positive numbers $c_1$ and $c_2$.

In particular,
$X_{A,m}:=A\hat{\otimes}_k X_m$ is the kernel of $R_{A,m}$, and
we have topological splitting $A_K=A_{k_m}\oplus X_{A,m}$.
For $\gamma\in \Gamma_k$, let $m(\gamma)\in\Z$
be the valuation of $\chi(\gamma)-1 \in \Z_p$.
Then there exists a positive integer $m(k)$ such that
for each $m\geq m(k)$ and $\gamma\in \Gamma_k$ with $m(\gamma)\leq m$,
$\gamma-1$ is invertible on $X_{A,m}$ and 
\[
 \val\bigl((\gamma-1)^{-1}a\bigr)\geq \val(a)-c_3
\]
for each $a\in A_K$.

Finally, for each matrix $U=(a_{ij})\in M_r(A_K)$, we set
$\val U:=\min_{i,j}\val a_{ij}$.

\begin{prop}\label{prop:finite Gamma orbit lies in A_infty}
Each finitely generated $A$-submodule of $A_K$ that is stable under the action of an open subgroup of $\Gamma_k$ is contained in $A_\infty$.
\end{prop}

\begin{proof}
We follow the proof of \cite[Proposition 3]{Sen-cont}.
By \cite[Corollaire 2.1.4]{Berger-Colmez},
there exist complete discrete valuation fields $E_1,\ldots E_s$ and
an isometric embedding $A\hra \prod_{i=1}^s E_i$.
Then extending the scalar yields an isometric embedding
\[
A_K=A_{k_m}\oplus X_{A,m}\hra
\prod_{i=1}^s E_i\hat{\otimes}_kK=\prod_{i=1}^s 
(E_i\hat{\otimes}_kk_m\oplus  E_i\hat{\otimes}_k X_m)
\]
preserving the topological splittings.

Let $\Gamma_k'$ be an open subgroup of $\Gamma_k$ and
 $W$  a finitely generated $A$-submodule of $A_K$ that is stable under the action of $\Gamma_k'$. 
Let $W_i$ be the finite-dimensional $E_i$-vector subspace of $E_i\hat{\otimes}_kK$
generated by the image of $W$ under the map $A_K\ra \prod_{i=1}^s E_i\hat{\otimes}_kK\ra E_i\hat{\otimes}_k K$.
To prove that $W$ is contained in $A_\infty=\bigcup_m A_{k_m}$,
it suffices to prove that for each $i$, there exists a large integer $m$
such that $W_i$ is contained in $E_i\hat{\otimes}_kk_m$.

Replacing $\Gamma_k'$ by a smaller open subgroup if necessary,
we may assume that there exists a topological generator $\gamma$ of $\Gamma_k'$.
Replacing $E_i$ by a finite field extension,
we may also assume that all the eigenvalues of the $E_i$-endomorphism $\gamma$
on $W_i$ lie in $E_i$.

Let $w\in W_i$ be an eigenvector for $\gamma$ and let $\lambda\in E$
be its eigenvalue.
Note that $\Gamma_k'$ acts continuously on $W_i$.
When $j$ goes to infinity, $\gamma^{p^j}$ approaches $1$
and thus $\lambda^{p^j}$ approaches $1$.
This implies that $\lambda$ is a principal unit, i.e. $\lvert \lambda-1\rvert_{E_i}<1$.

\begin{lem}
The eigenvalue $\lambda$ is a $p$-power root of unity.
\end{lem}

\begin{proof}
We follow the proof of \cite[Proposition 7(c)]{Tate}.
Assume the contrary.
We will prove that $\gamma-\lambda\colon E_i\hat{\otimes}_kK\ra E_i\hat{\otimes}_kK$ is bijective, which would contradict that the nonzero element $w\in W_i\subset E_i\hat{\otimes}_kK$ satisfies $(\gamma-\lambda)w=0$.

Let $m$ be the integer such that $k_m$ is 
the fixed subfield of $k_\infty$ by $\gamma$.
Consider the map $\gamma-1\colon E_i\hat{\otimes}_kK\ra E_i\hat{\otimes}_kK$.
This map preserves
the decomposition $E_i\hat{\otimes}_kK=E_i\hat{\otimes}_kk_m\oplus E_i\hat{\otimes}_k X_m$.
Moreover, it is zero on $E_i\hat{\otimes}_kk_m$ and bijective on $E_i\hat{\otimes}_k X_m$ with continuous inverse. Denote the inverse by $\rho$.
Then $\rho$ is a bounded $E_i\hat{\otimes}_kk_m$-linear operator with
operator norm at most $p^{c_3}$.
Since $\lambda\in E_i$ and $\lambda\neq 1$, the map $\gamma-\lambda$ is bijective on $E_i\hat{\otimes}_kk_m$. So it suffices to prove that $\gamma-\lambda$ is bijective on
$E_i\hat{\otimes}_k X_m$.

As operators on $E_i\hat{\otimes}_k X_m$, we have
\[
(\gamma-\lambda)\rho=\bigl((\gamma-1)-(\lambda-1)\bigr)\rho
=1-(\lambda-1)\rho.
\]
Thus if $\lvert \lambda-1\rvert_{E_i}p^{c_3}<1$,
then $1-(\lambda-1)\rho$ has an inverse on $E_i\hat{\otimes}_k X_m$
given by a geometric series and thus
$\gamma-\lambda$ admits a continuous inverse on $E_i\hat{\otimes}_k X_m$.
If $\lvert \lambda-1\rvert_{E_i}p^{c_3}\geq 1$,
first take a large integer $j$ with $\lvert \lambda^{p^j}-1\rvert_{E_i}p^{c_3}<1$. Then we can prove that $\gamma^{p^j}-\lambda^{p^j}$ has a bounded inverse 
on $E_i\hat{\otimes}_k X_m$. Hence so does $\gamma-\lambda$.
\end{proof}

We continue the proof of the proposition.
Since each eigenvalue of $\gamma$ on $W_i$ is a $p$-power root of unity,
we replace $\gamma$ by a higher $p$-power and may assume that
$\gamma$ acts on $W_i$ unipotently.
Thus $\gamma-1$ acts on $W_i$ nilpotently.

Let $m$ be the integer such that $k_m$ is the fixed subfield of $k_\infty$
by $\gamma$.
Then the map $\gamma-1\colon E_i\hat{\otimes}_kK\ra E_i\hat{\otimes}_kK$
is zero on $E_i\hat{\otimes}_kk_m$ and bijective on $E_i\hat{\otimes}_k X_m$.
This implies that the nilpotent endomorphism $\gamma-1$ on $W_i$ is actually zero and thus $W_i$ is contained in 
$E_i\hat{\otimes}_kk_m$.
\end{proof}

\begin{example}
For the trivial $A_K$-representation $V=A_K$ of $\Gamma_k$, 
we have $V_{\fin}=A_{\infty}$
by Proposition~\ref{prop:finite Gamma orbit lies in A_infty}.
\end{example}

The following theorem describes $V_{\fin}$
for a general $A_K$-representation $V$ of $\Gamma_k$,
and it was first proved by Sen (\cite{Sen-variation, Sen-infinite-dim}).

\begin{thm}\label{thm:Sen decompletion}
For an $A_K$-representation $V$ of $\Gamma_k$,
the $A_{\infty}$-module $V_{\fin}$
is finite free.
Moreover, the natural map
\[
 V_{\fin}\otimes_{A_{\infty}}A_K
\ra V
\]
is an isomorphism.
\end{thm}

\begin{proof}
First we prove the following lemma.
\begin{lem}
There exist an $A_K$-basis $v_1,\ldots,v_r\in V$ and a large positive integer
$m$
such that 
the transformation matrix of $\gamma$ with respect to this basis
has entries in $A_{k_m}$ for each $\gamma\in \Gamma_k$.
\end{lem} 

\begin{proof}
This follows from the Tate-Sen method for $\Gamma_k$-representations in the relative setting.
By \cite[Lemme 3.18]{Chenevier}, $V$ has a $\Gamma_k$-stable $A_K^\circ$-lattice.
Note that \cite[Lemme 3.18]{Chenevier} only concerns reduced affinoid algebras over a finite extension of $\Q_p$ but the same proof works for $A_K$ since one can apply Raynaud's theory to $A_K$.

By \cite[Corollaire 3.2.4]{Berger-Colmez}, there exist 
an $A_K$-basis $v_1,\ldots,v_r\in V$, a large positive integer
$m$, and an open subgroup $\Gamma_k'$ of $\Gamma_k$
such that 
the transformation matrix of $\gamma$ with respect to this basis
has entries in $A_{k_m}$ for each $\gamma\in \Gamma_k'$.
By increasing $m$ if necessary, we may also assume that $\Gamma_k'$ acts trivially on $A_{k_m}$.

For each $\gamma\in \Gamma_k$, we denote by $U_\gamma\in \GL_r(A_K)$
the transformation matrix of $\gamma$ with respect to $v_1,\ldots,v_r$.
Note that $U_{\gamma\gamma'}=U_\gamma\gamma(U_{\gamma'})$
for $\gamma, \gamma'\in \Gamma_k$.

Take a set $\{\gamma_1,\ldots,\gamma_s\}$ of coset representatives of $\Gamma_k/\Gamma_k'$
and let $W$ be the finitely generated $A_{k_m}$-submodule of $A_K$
generated by the entries of $U_{\gamma_1},\ldots, U_{\gamma_s}$.
Since $U_{\gamma_i\gamma'}=U_{\gamma_i}\gamma_i(U_{\gamma'})$
for $\gamma'\in \Gamma_k'$ and $\gamma_i(U_{\gamma'})$ has entries in $A_{k_m}$
by our construction, it follows that
$W$ is independent of the choice of the representatives $\gamma_1,\ldots,\gamma_s$.
Moreover, we have $\gamma'(U_{\gamma_i})=U_{\gamma'}^{-1}U_{\gamma'\gamma_i}$ for $\gamma'\in \Gamma_k'$.
From this we see that $W$ is stable under the action of $\Gamma_k'$.

Proposition~\ref{prop:finite Gamma orbit lies in A_infty} implies that
$W\subset A_\infty$, namely, $U_{\gamma_1},\ldots,U_{\gamma_s}\in \GL_r(A_\infty)$.
Thus if we increase $m$ so that $U_{\gamma_1},\ldots,U_{\gamma_s}\in \GL_r(A_{k_m})$, then $U_\gamma\in \GL_r(A_{k_m})$ for any $\gamma\in \Gamma_k$.
\end{proof}

We keep the notation in the proof of the lemma.
From the lemma, we see that $\bigoplus_{i=1}^r A_{\infty}v_i\subset V_{\fin}$.
So it suffices to prove that this is an equality.

Take any $v\in V_{\fin}$.
Let $W_v$ be the $A_{k_m}$-submodule of $A_K$ generated by
the coordinates of $\gamma v$ with respect to the basis $v_1,\ldots,v_r$
where $\gamma$ runs over all elements of $\Gamma_k$.
Since $v\in V_{\fin}$, this is a finitely generated $A_{k_m}$-module.

Write $v=\sum_{i=1}^r a_i v_i$ with $a_i\in A_K$
and denote the column vector of the $a_i$
by $\vec{a}$.
Then it is easy to see that $W_v$ is generated by
the entries of $U_{\gamma}\gamma(\vec{a})$ ($\gamma\in \Gamma_k$).
Since $U_{\gamma'\gamma}=U_{\gamma'}\gamma'(U_{\gamma})$
for $\gamma, \gamma'\in \Gamma_k$, we compute
\[
 \gamma'\bigl(U_{\gamma}\gamma(\vec{a})\bigr)
=U_{\gamma'}^{-1}U_{\gamma'\gamma}(\gamma'\gamma)(\vec{a}).
\]
 From this we see that $W_v$ is stable under the action of $\Gamma_k$.

By Proposition~\ref{prop:finite Gamma orbit lies in A_infty},
we have $W_v\subset A_\infty$.
In particular, $a_1,\ldots, a_r\in A_{\infty}$ and thus
$v\in \bigoplus_{i=1}^r A_{\infty}v_i$.
\end{proof}

\begin{prop}\label{prop:Fontaine decompletion for torsion}
Let $V$ be an $A\hat{\otimes}L_{\dR}^+$-representation of $\Gamma_k$.
If $V$ is finite free of rank $r$ over $A\hat{\otimes}L_{\dR}^+/(t^n)$,
then $V_{\fin}$
is finite free of rank $r$ over $A_{\infty}[t]/(t^n)$.
Moreover, the natural map
\[
 V_{\fin}\otimes_{A_{\infty}[[t]]}(A\hat{\otimes}L_{\dR}^+)
\ra V
\]
is an isomorphism. 
\end{prop}

\begin{proof}
We prove this proposition by induction on $n$.
When $n=1$, this is Theorem~\ref{thm:Sen decompletion}.
So we assume $n>1$.

Set $V':=t^{n-1}V$ and $V'':=V/V'$.
They are $A\hat{\otimes}L_{\dR}^+$-representations of $\Gamma_k$
and $V''$ is finite free of rank $r$ over 
$A\hat{\otimes}L_{\dR}^+/(t^{n-1})$.
By induction hypothesis, $V''_{\fin}$
is finite free of rank $r$ over $A_{\infty}[t]/(t^{n-1})$
and $V''_{\fin}\otimes_{A_{\infty}[t]/(t^{n-1})}A\hat{\otimes}L_{\dR}^+/(t^{n-1})\cong V''$.

Take lifts $v_1,\ldots,v_r$ of a basis of $V''_{\fin}$ to $V$.
Then $v_1,\ldots,v_r$ form an $A_{\infty}[t]/(t^n)$-basis of $V$.
We will prove that after a suitable modification of $v_1,\ldots,v_r$
the transformation matrix of $\gamma$ on $V$ with respect to the new basis
has entries in $A_{\infty}[t]/(t^n)$ for every $\gamma\in \Gamma_k$.

Suppose that we are given an element $\gamma$ of $\Gamma_k$.
For each $1\leq j\leq r$, write $\gamma v_j=\sum_{i=1}^r a_{ij} v_i$ with $a_{ij}\in A\hat{\otimes}L_{\dR}^+/(t^n)$.
Then the $r\times r$ matrix $T:=(a_{ij})$ 
is invertible since it is so modulo $t^{n-1}$. By the property of $V''_{\fin}$, we can write
\[
a_{ij}=a^0_{ij}+t^{n-1}a^1_{ij},\qquad
a^0_{ij}\in A_{\infty}[t]/(t^n),\quad  a^1_{ij}\in A_K=A\hat{\otimes}L_{\dR}^+/(t).
 \]
Set $U:=(a^{0}_{ij}\bmod t)\in M_r(A_\infty)$. This is invertible. In fact, $U$ is the transformation matrix of $\gamma$ acting on $V/tV$
with respect to the basis $(v_i \bmod t)$.

Since $\Gamma_k$ acts continuously on $V/tV$, 
$\val(U-1)>c_3$ and $m(\gamma)> \max\{c_3, m(k)\}$ for some $\gamma\neq 1$ close to $1$. From now on, we fix such $\gamma$.

\begin{claim}
There exists an element in $\GL_r(A\hat{\otimes}L_{\dR}^+/(t^n))$
of the form $1+t^{n-1}M$ with $M\in M_r(A_K)$ such that 
the $r\times r$ matrix
\[
 (1+t^{n-1}M)^{-1}T \gamma(1+t^{n-1}M)
\]
lies in $\GL_r(A_{\infty}[t]/(t^n))$.
\end{claim}

\begin{proof}
Noting that every element in $A\hat{\otimes}L_{\dR}^+/(t^n)$
is annihilated by $t^n$, we compute
\begin{align*}
  (1+t^{n-1}M)^{-1}T \gamma(1+t^{n-1}M)
&= (1-t^{n-1}M)T (1+\chi(\gamma)^{n-1}t^{n-1}\gamma(M))\\
&=T
-t^{n-1}\bigl(MT-\chi(\gamma)^{n-1}T\gamma(M)\bigr)\\
&\phantom{=}-t^{2(n-1)}\chi(\gamma)^{n-1}MT\gamma(M)\\
&=T-t^{n-1}\bigl(MU-\chi(\gamma)^{n-1}U\gamma(M)\bigr).
\end{align*}
Since $T=(a^0_{ij})+t^{n-1}(a^1_{ij})$ with $(a^0_{ij})\in \GL_r(A_{\infty}[t]/(t^n))$,
 it suffices to find $M\in M_r(A_K)$ such that
\[
 (a^1_{ij})-\bigl(MU-\chi(\gamma)^{n-1}U\gamma(M)\bigr)\in M_r(A_{\infty}).
\]

We will apply Lemma~\ref{lem:matrix equation} below to $U$, $U'=U^{-1}$ and $s=n-1$. Take $m\geq m(\gamma)$ large enough so that $U$ and $U^{-1}$ lie in $\GL_r(A_{k_m})$.
Recall the normalized trace map $R_{A,m}\colon A_K\ra A_{k_m}$ with the kernel $X_{A,m}$.
Since $R_{A,m}$ is $A_{k_m}$-linear, we see that 
$\bigl((1-R_{A,m})(a^1_{ij})\bigr)U^{-1}\in M_r(X_{A,m})$.
Therefore, by Lemma~\ref{lem:matrix equation}, there exists
$M_0\in M_r(X_{A,m})$ such that
\[
 \bigl((1-R_{A,m})(a^1_{ij})\bigr)U^{-1}
=M_0-\chi(\gamma)^{n-1}U\gamma(M_0)U^{-1}.
\]
From this we have
\[
 (a^1_{ij})-\bigl(M_0U-\chi(\gamma)^{n-1}U\gamma(M_0)\bigr)
= R_{A,m}(a^1_{ij}) \in M_r(A_{k_m}),
\]
and the matrix $1+t^{n-1}M_0$ satisfies the condition of the lemma.
\end{proof}

We continue the proof of the proposition.
We replace the basis $v_1,\ldots, v_r$ by the one corresponding to the matrix $1+t^{n-1}M$ in the lemma.
Then the transformation matrix of our fixed $\gamma$ 
with respect to the new $v_1,\ldots,v_r$ has entries in $A_{k_m}[t]/(t^n)$.
Thus for each $1\leq i\leq r$, 
the $\gamma^{\Z_p}$-orbit of $v_i$ is contained in 
a finitely generated $A_{k_m}[t]/(t^n)$-submodule of $V$
that is stable under $\gamma^{\Z_p}$.
Since $\gamma^{\Z_p}$ is of finite index in $\Gamma_k$,
the $\Gamma_k$-orbit of $v_i$ is also contained in 
a finitely generated $A_{k_m}[t]/(t^n)$-submodule of $V$ 
that is stable under $\Gamma_k$.
This means that $v_1,\ldots,v_r\in V_{\fin}$.
Hence $\bigoplus_{i=1}^r A_{\infty}[t]/(t^n)v_i\subset V_{\fin}$.

It remains to prove that $\bigoplus_{i=1}^r A_{\infty}[t]/(t^n)v_i= V_{\fin}$.
Since $A_{\infty}[t]/(t^n)\ra A\hat{\otimes}L_{\dR}^+/(t^n)$
is faithfully flat and $V=\bigoplus_{i=1}^r A\hat{\otimes}L_{\dR}^+/(t^n)v_i$, 
it is enough to show that the natural map 
$V_{\fin}\otimes_{A_{\infty}[t]/(t^n)} A\hat{\otimes}L_{\dR}^+/(t^n)\ra V$
is injective.
Note that 
$V_{\fin}\otimes_{A_{\infty}[t]/(t^n)} A\hat{\otimes}L_{\dR}^+/(t^n)
=V_{\fin}\otimes_{A_{\infty}[[t]]} A\hat{\otimes}L_{\dR}^+$.

Recall the exact sequence $0\ra V'\ra V\ra V''\ra 0$.
From this we have an exact sequence $0\ra V'_{\fin}\ra V_{\fin}\ra V''_{\fin}$, and it yields the following commutative diagram with exact rows 
\[
 \xymatrix{
0 \ar[r] 
&V'_{\fin}\otimes (A\hat{\otimes}L_{\dR}^+)\ar[r]\ar[d]
&V_{\fin}\otimes (A\hat{\otimes}L_{\dR}^+)\ar[r]\ar[d]
&V''_{\fin}\otimes (A\hat{\otimes}L_{\dR}^+)\ar[d]
&\\
0\ar[r]& V'\ar[r] &V\ar[r] &V''\ar[r] &0,
}
\]
where the tensor products in the first row are taken over $A_{\infty}[[t]]$.
By induction hypothesis, the first and the third vertical maps are isomorphisms.
Hence the second vertical map is injective and this completes the proof.
\end{proof}

The following lemma is used in the proof of Proposition~\ref{prop:Fontaine decompletion for torsion}.

\begin{lem}\label{lem:matrix equation}
Let $s$ be a positive integer.
Let $U, U'$ be elements in $ M_r(A_\infty)$
satisfying $\val(U-1)>c_3$ and $\val(U'-1)>c_3$.
Take a positive integer $m$ such that
$m> \max\{m(k),c_3\}$ and $U,U'\in M_r(A_{k_m})$.
Then for any $\gamma\in \Gamma_k$ with $c_3<m(\gamma)\leq m$,
the map
\[
 f\colon M_r(A_K)\ra M_r(A_K),\quad
M\mapsto M-\chi(\gamma)^sU\gamma(M)U'
\]
is bijective on the subset $M_r(X_{A,m})$
consisting of the $r\times r$ matrices with entries in 
the kernel $X_{A,m}$ of $R_{A,m}\colon A_K\ra A_{k_m}$.
\end{lem}

\begin{proof}
The proof of \cite[Lemma 15.3.9]{Brinon-Conrad} works in our setting.
For the convenience of the reader, we reproduce their proof here.

We first check that $f$ restricts to an endomorphism on $M_r(X_{A,m})$.
This follows from the fact that the map $R_{A,m}$ is $A_{k_m}$-linear and $\Gamma_k$-equivariant and thus 
$X_{A,m}$ is an $A_{k_m}$-module stable under the action of $\Gamma_k$.

We define a map $h\colon M_r(A_K)\ra M_r(A_K)$ by
\begin{align*}
  h(N)&:=N-\chi(\gamma)^sUNU'\\
&\phantom{:}=(N-\chi(\gamma)^sN)
+\chi(\gamma)^s\bigl((N-UN)+UN(1-U')\bigr).
\end{align*}
Then the same argument as above shows that $h$ restricts
 to an endomorphism on $M_r(X_{A,m})$.
We also have $f(M)=(1-\gamma)M +h(\gamma M)$.

Recall that the map $1-\gamma\colon M_r(X_{A,m})\ra M_r(X_{A,m})$
admits a continuous inverse with the operator norm at most $p^{c_3}$.
We denote this inverse by $\rho$.
Since $(f\circ \rho-\id)M=h(\gamma \rho(M))$,
it suffices to prove that the operator norm of $h$ 
is less than $p^{-c_3}$; this would imply that the operator norm of $h\circ \gamma\circ \rho$ 
is less than $1$.
Thus $f\circ \rho$ admits a continuous inverse given by a geometric series and hence $f$ is bijective on $M_r(X_{A,m})$.

By the second expression of $h$, we have
\begin{align*}
  \val(h(N))
&\geq \min\{\val((1-\chi(\gamma)^s)N), \val((U-1)N), \val(UN(1-U'))\}\\
&\geq \min\{\val((1-\chi(\gamma))N), \val((U-1)N), \val(N(1-U'))\}.
\end{align*}
From this we have
\[
 \val(h(N))\geq \val(N)+\delta
\]
where $\delta:=\min\{m(\gamma),\val(U-1),\val(U'-1)\}$.
Thus the operator norm of $h$ is at most $p^{-\delta}$.
Since $\delta>c_3$ by assumption, this completes the proof.
\end{proof}

\begin{proof}[Proof of Theorem~\ref{thm:Fontaine decompletion}]
 For each $n\geq 1$, put $V_n:=V/t^nV$.
This is an $A\hat{\otimes}L_{\dR}^+$-representation of $\Gamma_k$
that is finite free of rank $r$ over $A\hat{\otimes}L_{\dR}^+/(t^n)$.
Thus by Proposition~\ref{prop:Fontaine decompletion for torsion}, 
$(V_{n})_{\fin}$
is finite free of rank $r$ over $A_{\infty}[t]/(t^n)$, and
$(V_{n})_{\fin}\otimes_{A_{\infty}[[t]]}(A\hat{\otimes}L_{\dR}^+)\ra V_n$
is an isomorphism. 

By definition, we have $V_{\fin}=\varprojlim_n (V_{n})_{\fin}$.
Since the natural map $V_{n+1}\ra V_n$ is surjective,
so is the map $(V_{n+1})_{\fin}\ra (V_{n})_{\fin}$
by the faithfully flatness of $A_\infty[t]/(t^{n+1})\ra A\hat{\otimes}L_{\dR}^+/(t^{n+1})$.
Thus lifting a basis of $(V_{n})_{\fin}$ gives a basis of $V_{\fin}$
and we see that $V_{\fin}$ is finite free of rank $r$ over $A_\infty[[t]]$.
The remaining assertions also follow from this.
\end{proof}

\begin{prop}\label{prop:fully faithfulness}
 For an $A\hat{\otimes}L_{\dR}^+$-representation $V$ of $\Gamma_k$
that is finite free of rank $r$ over $A\hat{\otimes}L_{\dR}^+$,
the $A_\infty[[t]]$-module $V_{\fin}$
is the union of finitely generated $A_\infty[[t]]$-submodules of $V$
that are stable under the action of $\Gamma_k$.
In particular, the natural inclusion
\[
 (V_{\fin})^{\Gamma_k}\hra V^{\Gamma_k}
\]
is an isomorphism.
\end{prop}

\begin{proof}
Let $V_{\fin}'$ denote the union of finitely generated $A_\infty[[t]]$-submodules of $V$ that are stable under the action of $\Gamma_k$.
Then $V_{\fin}\subset V_{\fin}'$ by Theorem~\ref{thm:Fontaine decompletion}.
So it remains to prove the opposite inclusion.
For this it suffices to prove $V_{\fin}'/t^nV_{\fin}'\subset V_{\fin}/t^nV_{\fin}$ for each $n\geq 1$.
Since $V_{\fin}/t^nV_{\fin}=(V/t^nV)_{\fin}$ 
by Theorem~\ref{thm:Fontaine decompletion},
the desired inclusion follows from the definition of $(V/t^nV)_{\fin}$
noting $A_\infty[t]/(t^n)=\bigcup_m A_{k_m}[t]/(t^n)$.
The second assertion follows from the first.
\end{proof}

\begin{example}
For the trivial $A\hat{\otimes}L_{\dR}^+$-representation 
$V=A\hat{\otimes}L_{\dR}^+$ of $\Gamma_k$, 
we have $V_{\fin}=A_{\infty}[[t]]$.
\end{example}

Finally, we discuss topologies on $V_{\fin}$ and 
the continuity of the action of $\Gamma_k$.

\begin{lem}\label{lem:topology and continuity}
Let $W$ be a finite free $A_{\infty}[[t]]/(t^n)$-module equipped with an action of $\Gamma_k$.
Then $\Gamma_k$-action is continuous with respect to the topology on $W$ induced from the product topology on $A_\infty[[t]]/(t^n)\cong A_\infty^n$
if and only if it is continuous with respect to 
the topology on $W$ induced from the subspace topology on 
$A_\infty[[t]]/(t^n)\subset A\hat{\otimes}L_{\dR}^+/(t^n)$.
\end{lem}

\begin{proof}
For each of the two topologies on $W$, the continuity of $\Gamma_k$ implies that
there exist an $A_\infty[[t]]/(t^n)$-basis $w_1,\ldots, w_r$ of $W$
and a large positive integer $m$ such that
$W_m:=\bigoplus_{i=1}^r A_{k_m}[[t]]/(t^n) w_i$ is stable under $\Gamma_k$
and its action on $W_m$ is continuous with respect to the induced topology 
$W_m\subset W$.
Conversely, if the $\Gamma_k$-action on $W_m$ is continuous 
with respect to the induced topology $W_m\subset W$
 for such $\Gamma_k$-stable $A_{k_m}[[t]]/(t^n)$-submodule $W_m$
with $W_m\otimes_{A_{k_m}[[t]]/(t^n)}A_\infty[[t]]/(t^n)=W$,
the $\Gamma_k$-action on $W$ is continuous.
 
The subspace topology on $A_{k_m}[[t]]/(t^n)$ from $A\hat{\otimes}L_{\dR}^+/(t^n)$ coincides with the product topology 
on $A_{k_m}[[t]]/(t^n)\cong A_{k_m}^n$.
From this we find that the continuity conditions on the action of $\Gamma_k$ on $W_m$ with respect to the two topologies coincide.
Hence the two continuity properties of the action of $\Gamma_k$ on $W$ are equivalent.
\end{proof}

\begin{defn}
 Let $V$ be an $A\hat{\otimes}L_{\dR}^+$-representation of $\Gamma_k$.
\begin{itemize}
 \item If $V$ is finite free over $A\hat{\otimes}L_{\dR}^+/(t^n)$ for some $n\geq 1$,
we equip $V_{\fin}$ with the topology acquired from topologizing
$A_{\infty}[[t]]/(t^n)$ with the product topology of the $p$-adic topology on $A_\infty$. 
Then $\Gamma_k$ acts continuously on $V_{\fin}$
by Lemma~\ref{lem:topology and continuity}.
 \item If $V$ is finite free over $A\hat{\otimes}L_{\dR}^+$,
we equip $V_{\fin}$ with the inverse limit topology
via $V_{\fin}=\varprojlim_n (V/t^nV)_{\fin}$.
Then $\Gamma_k$ acts continuously on $V_{\fin}$.
\end{itemize}
\end{defn}

\begin{defn}
An \textit{$A_\infty[[t]]$-representation} of $\Gamma_k$ 
is an $A_\infty[[t]]$-module $W$ 
that is isomorphic to either $(A_\infty[[t]])^r$
or $(A_\infty[[t]]/(t^n))^r$ for some $r$ and $n$,
equipped with a continuous
$A_\infty[[t]]$-semilinear action of $\Gamma_k$
(here the topology on $W$ is acquired from the $p$-adic topology on $A_\infty$
by considering the product topology and the inverse limit topology as before).
We denote the category of $A_\infty[[t]]$-representations of $\Gamma_k$ 
by $\mathrm{Rep}_{\Gamma_k}(A_{\infty}[[t]])$.
An $A_\infty[[t]]$-representation of $\Gamma_k$
that is annihilated by $t$ is also called
an \textit{$A_\infty$-representation} of $\Gamma_k$.
\end{defn}

\begin{thm}
The decompletion functor
\[
 \mathrm{Rep}_{\Gamma_k}(A\hat{\otimes}L_{\dR}^+)
\ra \mathrm{Rep}_{\Gamma_k}(A_\infty[[t]]),
\quad V\mapsto V_{\fin}
\]
is an equivalence of categories.
A quasi-inverse is given by
$W\mapsto W\otimes_{A_\infty[[t]]}(A\hat{\otimes}L_{\dR}^+)$.
\end{thm}

\begin{proof}
By Theorem~\ref{thm:Fontaine decompletion}, Proposition~\ref{prop:Fontaine decompletion for torsion}, and Lemma~\ref{lem:topology and continuity}, the functor is well-defined and essentially surjective.
The full faithfulness follows from 
Proposition~\ref{prop:fully faithfulness}.
\end{proof}

\subsection{Sen's endomorphism and Fontaine's connection in the relative setting}
\label{subsection:Sen's operator and Fontaine's connection}

\begin{prop}\label{prop:Sen endom}
Let $W$ be an $A_{\infty}$-representation of $\Gamma_k$.
Then there exists a unique  $A_{\infty}$-linear map
$\phi_{W}\colon W\ra W$ satisfying the following property:
For any $w\in W$, there exists
an open subgroup $\Gamma_{k,w}$ of $\Gamma_k$ such that
\[
 \gamma w=\exp (\log(\chi(\gamma))\phi_{W})(w)
\]
for $\gamma\in \Gamma_{k,w}$.
Here $\log$ (resp.~$\exp$)
is the $p$-adic logarithm (resp. exponential).
Moreover, $\phi_{W}$ is $\Gamma_k$-equivariant and 
functorial with respect to $W$.
\end{prop}

\begin{rem}
The proposition says that the endomorphism $\phi_{W}$
is computed as 
\[
 \phi_{W}(w)=\lim_{\gamma \to 1}
\frac{\gamma w-w}{\log \chi(\gamma)}
\]
for $w\in W$.
\end{rem}

\begin{proof}
 This is standard;
arguments in \cite[Theorem 4]{Sen-cont} also
work in our setting.
See also \cite[\S 2]{Sen-infinite-dim}, \cite[Proposition 4]{Sen-variation},
\cite[Proposition 2.5]{Fontaine}, and \cite[\S 15.1]{Brinon-Conrad}.
\end{proof}

The following lemma is also proved by standard arguments.

\begin{lem}\label{lem:properties of Sen endom}
 Let $W_1$ and $W_2$ be
$A_{\infty}$-representations of $\Gamma_k$.
Then we have the following equalities:
\begin{itemize}
 \item $\phi_{W_1\oplus W_2}=\phi_{W_1}\oplus \phi_{W_2}$ on $W_1\oplus W_2$.
 \item $\phi_{W_1\otimes W_2}=\phi_{W_1}\otimes \phi_{W_2}$ on $W_1\otimes W_2$.
 \item $\phi_{\Hom(W_1,W_2)}(f)=\phi_{W_2}\circ f - f\circ \phi_{W_2}$
for $f\in \Hom(W_1,W_2)$.
\end{itemize}
\end{lem}

\begin{defn}\label{defn:base chage of Sen endom}
Let $V$ be an $A_K$-representation of $\Gamma_k$.
We denote by $\phi_V$ the $A_K$-linear endomorphism $\phi_{V_{\fin}}\otimes \id_{A_K}$ on $V=V_{\fin}\otimes_{A_\infty}A_K$.
\end{defn}

\begin{prop}\label{prop:Fontaine connection}
Let $W$ be an $A_{\infty}[[t]]$-representation of $\Gamma_k$. Then there exists a unique $A_{\infty}$-linear map
$\phi_{\dR,W}\colon W\ra W$ 
satisfying the following property:
For each $n\in\N$ and $w\in W$, there exists
an open subgroup $\Gamma_{k,n,w}$ of $\Gamma_k$ such that
\[
 \gamma w\equiv \exp (\log (\chi(\gamma))\phi_{\dR,W})(w)  
\pmod{t^n W}
\]
for $\gamma \in \Gamma_{k,n,w}$.
\end{prop}

\begin{proof}
Note that $W/t^n W$ is an $A_\infty$-representation of $\Gamma_k$.
So the proposition follows from Proposition~\ref{prop:Sen endom}.
\end{proof}

\begin{defn}
Set $A_{\infty}((t)):=A_{\infty}[[t]][t^{-1}]$.
We denote by $\partial_t$ the $A_{\infty}$-linear endomorphism
\[
 A_{\infty}((t))\ra A_{\infty}((t)),
\quad \textstyle\sum_{j\gg-\infty}a_jt^j\mapsto 
\textstyle\sum_{j\gg-\infty}ja_jt^{j-1}.
\]
The restriction of $\partial_t$ to $A_{\infty}[[t]]$ is also denoted by $\partial_t$.
\end{defn}

\begin{prop}\label{prop:properties of Fontaine's connection}
For an $A_{\infty}[[t]]$-representation $W$ of $\Gamma_k$,  
the endomorphism $\phi_{\dR,W}\colon W\ra W$ satisfies
\[
 \phi_{\dR,W}(\alpha w)=t\partial_t(\alpha)w+\alpha\phi_{\dR,W}(w)
\]
for every $\alpha\in A_{\infty}[[t]]$ and $w\in W$.
\end{prop}

\begin{proof}
By the characterizing property of $\phi_{\dR,W}$,
we may assume that $W$ is annihilated by some power of $t$.
In this case, it is enough to check the equality for $\alpha=t^j$
by $A_\infty$-linearity of $\phi_{\dR,W}$.
By induction on $j$, we may further assume that $\alpha=t$.

So we need to show 
$\phi_{\dR,W}(t w)=t w+t\phi_{\dR,W}(w)$.
This follows from
\begin{align*}
 \phi_{\dR, W}(t w)
&=\lim_{\gamma \to 1}
\frac{\gamma (t w)-t w}{\log \chi(\gamma)}\\
&=\lim_{\gamma \to 1}
\frac{\chi(\gamma)-1}{\log \chi(\gamma)}t\gamma(w)
+t\lim_{\gamma \to 1}\frac{\gamma (w)-w}{\log \chi(\gamma)}\\
&=t w+t\phi_{\dR,W}(w).
\end{align*}

\end{proof}

\begin{lem-defn}
Let $W$ be a finite free $A_{\infty}[[t]]$-representation of $\Gamma_k$. Then $W[t^{-1}]:=W\otimes_{A_{\infty}[[t]]}A_{\infty}((t))$
is a finite free $A_{\infty}((t))$-representation of $\Gamma_k$
with $\Gamma_k$-stable decreasing filtration
$\Fil^j W[t^{-1}]:=t^j W$.
Moreover, the $A_{\infty}$-linear endomorphism $\phi_{\dR, W[t^{-1}]}\colon W[t^{-1}]\ra W[t^{-1}]$ sending $w\in \Fil^{j}W[t^{-1}]$ to
\[
 \phi_{\dR,W[t^{-1}]}(w):=jw+t^j\phi_{\dR,W}(t^{-j}w)
\]
is well-defined and satisfies $\phi_{\dR,W[t^{-1}]}|_{W}=\phi_{\dR,W}$.
\end{lem-defn}

\begin{proof}
This follows from Proposition~\ref{prop:properties of Fontaine's connection}.
\end{proof}

\begin{defn}\label{defn:de Rham connection}
Let $V$ be a finite free filtered $A\hat{\otimes}L_{\dR}$-representation
of $\Gamma_k$.
Define
\[
 V_{\fin}:=(\Fil^0 V)_{\fin}[t^{-1}].
\]
Then $V_{\fin}$ is 
 a finite free $A_{\infty}((t))$-representation of $\Gamma_k$
equipped with $\Gamma_k$-stable decreasing filtration and 
$\phi_{\dR, V_{\fin}}$.
Since $\phi_{\dR, V_{\fin}}$ preserves the filtration,
it defines an $A_{\infty}$-linear endomorphism on 
$\gr^0 V_{\fin}$, which we denote by $\Res_{\Fil^0 V_{\fin}}\phi_{\dR, V_{\fin}}$.
it follows from definition that
\[
\Res_{\Fil^0 V_{\fin}}\phi_{\dR, V_{\fin}}
=\phi_{\gr^0 V} 
\]
as endomorphisms on the finite free $A_\infty$-module
$\gr^0 (V_{\fin})= (\gr^0 V)_{\fin}$.
\end{defn}

\section{The arithmetic Sen endomorphism of a $p$-adic local system}
\label{section:arithmetic Sen endomorphism}

From this section, we study the relative $p$-adic Hodge theory in geometric families.
Let $k$ be a finite field extension of $\Q_p$ and 
let $X$ be an $n$-dimensional smooth rigid analytic variety over $\Spa(k,\calO_k)$.
Let $K$ be the $p$-adic completion of $k_\infty:=\bigcup_n k(\mu_{p^n})$ and
let $X_K$ denote the base change of $X$ to $\Spa (K,\calO_K)$.
We denote by $\Gamma_k$ the Galois group $\Gal(k_\infty/k)$.

Based on the recent progresses on the relative $p$-adic Hodge theory 
\cite{KL-I, KL-II, Scholze-perfectoid, Scholze}, Liu and Zhu attached to an \'etale $\Q_p$-local system $\bL$ 
a nilpotent Higgs bundle $\calH(\bL)$ on $X_K$ equipped with $\Gamma_k$-action
(\cite{LZ}).
Our goal is to define an endomorphism $\phi_{\bL}$ on $\calH(\bL)$ 
by decompleting $\Gamma_k$-action. 
The endomorphism $\phi_{\bL}$, 
which we will call the arithmetic Sen endomorphism, 
is a natural generalization of the Sen endomorphism of a $p$-adic Galois representation of $k$.

\subsection{Review of the $p$-adic Simpson correspondence \`a la Liu and Zhu}

First let us briefly recall the sites and sheaves that we use.
Let $X_{\proet}$ be the pro-\'etale site on $X$ in the sense of 
\cite{Scholze, Scholze-errata}.
The pro-\'etale site is equipped with a natural projection to the \'etale site on $X$
\[
\nu\colon X_{\proet}\ra X_{\et}.
\]
Let $\nu'\colon X_{\proet}/X_K\ra (X_K)_{\et}$ be the restriction of $\nu$
and we identify $X_{\proet}/X_K$ with $(X_K)_{\proet}$ (see 
a discussion before Proposition 6.10 in \cite{Scholze}).

We denote by $\hat{\Z}_p$ (resp. $\hat{\Q}_p$) the constant sheaf on $X_{\proet}$ associated to $\Z_p$ (resp. $\Q_p$). For a $\Z_p$-local system $\bL$ (resp. $\Q_p$-local system) on $X_{\et}$, let $\hat{\bL}$ denote the $\hat{\Z}_p$-module
(resp. $\hat{\Q}_p$-module) on $X_{\proet}$ associated to $\bL$ (see \cite[\S 8.2]{Scholze}).

We define sheaves on $X_{\proet}$ as follows.
We set 
\[
 \calO_X^+:=\nu^\ast \calO_{X_{\et}}^+, \quad
\calO_X:=\nu^\ast \calO_{X_{\et}}, \quad \text{and}
\quad \hat{\calO}_X:=\bigl(\varprojlim_n \calO_X^+/p^n\bigr)[p^{-1}].
\]
We also set $\Omega_X^1=\nu^\ast \Omega_{X_{\et}}^1$ and we denote its $i$-th exterior power by $\Omega_X^i$.
Moreover, Scholze introduced the de Rham period sheaves $\bB_{\dR}^+$, 
$\bB_{\dR}$, $\OB_{\dR}^+$ and $\OB_{\dR}$ on $X_{\proet}$ in \cite[\S 6]{Scholze} and \cite{Scholze-errata}.
The structural de Rham sheaf $\OB_{\dR}$ has the following properties:
it is a sheaf of $\calO_X$-algebras equipped with
a decreasing filtration $\Fil^\bullet \OB_{\dR}$ and an integrable connection
\[
 \nabla\colon \OB_{\dR}\ra \OB_{\dR}\otimes_{\calO_X}\Omega_X^1
\]
satisfying the Griffiths transversality.
Since $X$ is assumed to be smooth of dimension $n$, 
this gives rise to the following exact sequence of sheaves on $X_{\proet}$:
\[
 0 \ra \bB_{\dR}\lra \OB_{\dR}\stackrel{\nabla}{\lra}
\OB_{\dR}\otimes_{\calO_X}\Omega_X^1\stackrel{\nabla}{\lra}\cdots
\OB_{\dR}\otimes_{\calO_X}\Omega_X^n \lra 0.
\]

Finally, we set $\OC:=\gr^0\OB_{\dR}$. Taking the associated graded connection
of $\nabla$ on $\OB_{\dR}$ equips $\OC$ with a Higgs field 
\[
 \gr^0 \nabla \colon \OC\ra \OC \otimes_{\calO_X}\Omega_X^1(-1),
\]
where $(-1)$ stands for the $(-1)$st Tate twist.

We review the formulation of the $p$-adic Simpson correspondence 
by Liu and Zhu.
Let $\bL$ be a $\Q_p$-local system on $X_{\et}$ of rank $r$.
We define
\[
 \calH(\bL)
=\nu'_\ast\bigl(\hat{\bL}\otimes_{\hat{\Q}_p}\OC\bigr).
\]
Then Liu and Zhu proved the following theorem.

\begin{thm}[Rough form of {\cite[Theorem 2.1]{LZ}}]\label{thm:Simpson by LZ}
 $\calH(\bL)$ is a vector bundle on $X_K$ of rank $r$
equipped with a nilpotent Higgs field $\vartheta_{\bL}$
and a semilinear action of $\Gamma_k$.
The functor $\calH$ is a tensor functor
from the category of $\Q_p$-local systems on $X_{\et}$
to the category of nilpotent Higgs bundles on $X_K$.
Moreover, $\calH$ is compatible with pullback and (under some conditions)
smooth proper pushforward. 
\end{thm}

\begin{rem}
For our purpose, we use the $p$-adic Simpson correspondence 
formulated by Liu and Zhu
as their output is a Higgs bundle over $X_K$ with 
a $\Gamma_k$-action.
See \cite{Faltings-Simpson} and \cite{AGT} 
for the $p$-adic Simpson correspondence
by Faltings and Abbes-Gros-Tsuji
in a more general setting, and 
see \cite{AB-surconv, AB} for the one
over a pro-\'etale cover of $X_K$
 by Andreatta and Brinon.
\end{rem}

To define the arithmetic Sen endomorphism on $\calH(\bL)$ and discuss its properties,
let us recall Liu and Zhu's arguments in the proof of Theorem~\ref{thm:Simpson by LZ}.

We follow the notation on base changes of adic spaces and rings
in \cite{LZ}.
We denote by $\bT^n$ the $n$-dimensional rigid analytic torus
\[
 \Spa (k\langle T_1^{\pm},\ldots, T_n^{\pm}\rangle,
\calO_k\langle T_1^{\pm},\ldots, T_n^{\pm}\rangle).
\]
For $m\geq 0$, we set
\[
 \bT^n_m=\Spa (k_m\langle T_1^{\pm 1/p^m},\ldots, T_n^{\pm 1/p^m}\rangle,
\calO_{k_m}\langle T_1^{\pm 1/p^m},\ldots, T_n^{\pm 1/p^m}\rangle).
\]
We denote by $\tilde{\bT}^n_\infty$ the affinoid perfectoid
$\varprojlim_m \bT^n_m$ in $X_{\proet}$.

To study properties of $\calH(\bL)$, 
we introduce the following base $\calB$ for $(X_K)_{\et}$:
objects of $\calB$ are
the \'etale maps to $X_K$ that are the base changes 
of standard \'etale morphisms $Y\ra X_{k'}$
defined over some finite extension $k'$ of $k$ in $K$
where $Y$ is affinoid admitting a toric chart after some finite 
extension of $k'$.
Recall that an \'etale morphism between adic spaces
is called standard \'etale if it is a composite of 
rational localizations and finite \'etale morphisms
and that a toric chart means a standard \'etale morphism to $\bT^n$.
Morphisms of $\calB$ are the base changes of 
\'etale morphisms over some finite extension of $k$ in $K$.
We equip $\calB$ with the induced topology from $(X_K)_{\et}$.
Then the associated topoi $(X_K)_{\et}^\sim$
and $\calB^\sim$ are equivalent (\cite[Lemma 2.5]{LZ}).

When $Y=\Spa (B,B^+)$ admits a toric chart over $k$, 
we use the following notation:
we set
\[
 Y_m=\Spa (B_m, B_m^+) := Y\times_{\bT^n}\bT^n_m.
\]
Then $\tilde{Y}_\infty:= Y\times_{\bT^n}\tilde{\bT}^n_\infty$
is the affinoid perfectoid in $Y_{\proet}$
represented by the relative toric tower $(Y_n)$.
We denote by $(\hat{B}_\infty, \hat{B}^+_\infty)$
the perfectoid affinoid completed direct limit of the $(B_m,B^+_m)$'s
and set $\hat{Y}_\infty:=\Spa (\hat{B}_\infty, \hat{B}^+_\infty)$,
the affinoid perfectoid space associated to $Y_\infty$.
We also set $B_{k_m}=B\otimes_k k_m$ as in Subsection~\ref{subsection:setup}.
When $Y$ admits a toric chart over a finite extension of $k$ in $K$,
we similarly define these objects using the rigid analytic torus over the field.

Let $Y_{K,m}:=\Spa (B_{K,m}, B_{K,m}^+)$
be the base change of $Y_m$ from $k_m$ to $K$
and let $\tilde{Y}_{K,\infty}$
be the affinoid perfectoid represented by the toric tower $(Y_{K,m})$.
We denote the associated affinoid perfectoid space
by $\hat{Y}_{K,\infty}=\Spa (\hat{B}_{K,m},\hat{B}_{K,m}^+)$.
The cover $\tilde{Y}_{K,\infty}/Y$ is Galois.
We denote its Galois group by $\Gamma$.
Then $\Gamma$ fits into a splitting exact sequence
\[
 1\ra \Gamma_{\geom}\ra \Gamma \ra \Gamma_k\ra 1.
\]

To prove Theorem~\ref{thm:Simpson by LZ}, 
Liu and Zhu gave a simple description of 
\[
 \calH(\bL)(Y_K)=H^0(X_{\proet}/Y_K,\hat{\bL}\otimes\OC)
\]
for $\bigl(Y=\Spa(B,B^+)\ra X_{k'}\bigr)\in\calB$,
which we recall now.

\begin{prop}[{\cite[Proposition 2.8]{LZ}}]\label{prop:def of M_K(Y)}
Put $\calM=\hat{\bL}\otimes_{\hat{\Q}_p}\hatO_X$.
Then there exists a unique finite projective $B_K$-submodule $M_K(Y)$ of $\calM(\tilde Y_{K,\infty})$, which is stable under $\Gamma$, such that
\begin{enumerate}
\item $M_K(Y)\otimes_{B_K}\hat{B}_{K,\infty}=\calM(\tilde Y_{K,\infty})$;
\item The $B_K$-linear representation of $\Gamma_{\geom}$ on $M_K(Y)$ is unipotent;
\end{enumerate}
In addition, the module $M_K(Y)$ has the following properties:
\begin{enumerate}
\item[(P1)] There exist some positive integer $j_0$ and some finite projective $B_{k_{j_0}}$-submodule $M(Y)$ of $M_K(Y)$ stable under $\Gamma$ such that $M(Y)\otimes_{B_{k_{j_0}}}B_K=M_K(Y)$. Moreover, the construction of $M(Y)$ is compatible with base change along standard \'etale morphisms.
\item[(P2)] The natural map
\[
 M_K(Y)^{\Gamma_{\geom}} \ra \calM(\tilde Y_{K,\infty})^{\Gamma_{\geom}}
\]
is an isomorphism.
 \end{enumerate}
\end{prop}

Once this proposition is proved, we can describe 
$\calH(\bL)(Y_K)$ in terms of $M_K(Y)$ as follows:
the vanishing theorem on  affinoid perfectoid spaces 
(\cite[Proposition 7.13]{Scholze-perfectoid})
implies the degeneration of 
the Cartan-Leray spectral sequence to the Galois cover 
$\{\tilde{Y}_{K,\infty}\ra Y_K\}$ with Galois group $\Gamma_{\geom}$,
and thus we have
\begin{align*}
 H^i(\Gamma_{\geom},\calM(\tilde{Y}_{K,\infty}))&\stackrel{\cong}{\lra}
H^i(X_{\proet}/Y_K,\calM),\\
 H^i(\Gamma_{\geom},(\calM\otimes \OC)(\tilde{Y}_{K,\infty}))&\stackrel{\cong}{\lra}
H^i(X_{\proet}/Y_K,\calM\otimes\OC).
\end{align*}
Moreover, we know that 
$\OC|_{\tilde{Y}_{K,\infty}}\cong(\hatO_X|_{\tilde{Y}_{K,\infty}})[V_1,\ldots,V_n]$,
where $V_i=t^{-1}\log ([T_i^\flat]/T_i)$
for a fixed compatible sequence of $p$-power roots of
the coordinate $T_i^\flat=(T_i,T_i^{1/p},\ldots)$.
It follows from these results and a simple argument on the direct limit of sheaves on $X_{\proet}$ that the natural $\Gamma_k$-equivariant map
\[
 \bigl(M_K(Y)[V_1,\ldots,V_n]\bigr)^{\Gamma_{\geom}}
\ra \calH(\bL)(Y_K)
\]
is an isomorphism.
A simple computation shows that
the map $M_K(Y)[V_1,\ldots,V_n]\ra M_K(Y)$ sending $V_i$ to $0$
induces a $\Gamma_k$-equivariant isomorphism
\[
\bigl(M_K(Y)[V_1,\ldots,V_n]\bigr)^{\Gamma_{\geom}}
\stackrel{\cong}{\lra}M_K(Y).
\]
Thus we have a $\Gamma_k$-equivariant isomorphism
\[
 \calH(\bL)(Y_K)\cong M_K(Y).
\]
The above discussion is summarized in the following commutative diagram:
\[
 \xymatrix{
M_K(Y)\otimes_{B_K}\hat{B}_{K,\infty} \ar@{^{(}->}[r]
& (M_K(Y)\otimes_{B_K}\hat{B}_{K,\infty})[V_1,\ldots,V_n]
&\\
M_K(Y) \ar@{^{(}->}[u]\ar@{^{(}->}[r]
& M_K(Y)[V_1,\ldots,V_n] \ar@{^{(}->}[u]\ar[r]^{V_i=0}
& M_K(Y)\\
M_K(Y)^{\Gamma_{\geom}} \ar@{^{(}->}[u]\ar@{^{(}->}[r]
& \bigl(M_K(Y)[V_1,\ldots,V_n]\bigr)^{\Gamma_{\geom}} 
\ar@{^{(}->}[u]\ar[r]^{\phantom{aaaaa}\cong}
& M_K(Y).\ar@{=}[u]\\
 }
\]

Finally, we recall the Higgs field $\vartheta_{\bL}$.
This is defined to be
\[
 \vartheta_{\bL}:=\nu_\ast'\bigl(\gr\nabla\colon
\hat{\bL}\otimes\OC\ra\hat{\bL}\otimes\OC\otimes\Omega_X^1(-1)\bigr)
\]
under the identification
$\nu'_\ast\bigl(\hat{\bL}\otimes\OC\otimes\Omega_X^1(-1)\bigr)
\cong \calH(\bL)\otimes\Omega_{X/k}^1(-1)$.
Here $\calH(\bL)\otimes\Omega_{X/k}^1(-1)$
denotes the $\calO_{X_K}$-module
$\calH(\bL)\otimes_{\calO_X}\Omega_{X/k}^1(-1)
=\calH(\bL)\otimes_{\calO_{X_K}}\Omega_{X_K/K}^1(-1)$
equipped with a natural $\Gamma_k$-action.

We have another description under 
the isomorphism
$\calH(\bL)(Y_K)\cong M_K(Y)$, which proves $\vartheta_{\bL}$ is nilpotent.
Namely, let $\rho_{\geom}$ denote
the action of $\Gamma_{\geom}$ on $M_K(Y)$ and
let $\chi_i\colon \Gamma_{\geom}\cong \Z_p(1)^n\ra \Z_p(1)$ 
denote the composite of the natural identification and projection to the $i$-th component.
We can take the logarithm of $\rho_{\geom}$ on $M(Y)\subset M_K(Y)$
since the action is unipotent.
Suppose 
the logarithm is written as
\[
 \log \rho_{\geom} 
=\sum_{i=1}^n \vartheta_i\otimes \chi_i\otimes t^{-1}, 
\]
where $\vartheta_i\in \End(M(Y))$.
Then $\vartheta_i$ can be regarded as an endomorphism on $M_K(Y)$ by extension of scalars and we define 
\begin{equation}\label{eq:Higgs field of M_K(Y)}
 \vartheta_{M_K(Y)}:=\sum_{i=1}^n \vartheta_i\otimes d\log T_i\otimes t^{-1}
=\sum_{i=1}^n \vartheta_i\otimes \frac{dT_i}{T_i}\otimes t^{-1}
\in \End(M_K(Y))\otimes_B\Omega_{B/k'}^1(-1). 
\end{equation}
We can check $\vartheta_{M_K(Y)}\wedge\vartheta_{M_K(Y)}=0$
and this defines a Higgs field on $M_K(Y)$.
It turns out that $\vartheta_{\bL}(Y_K)=\vartheta_{M_K(Y)}$
under the $\Gamma_k$-equivariant isomorphism
$\calH(\bL)(Y_K)\cong M_K(Y)$.
See \cite[\S 2]{LZ} for the detail.

\subsection{Definition and properties of the arithmetic Sen endomorphism}
We will define the arithmetic Sen endomorphism
$\phi_{\bL}\in\End \calH(\bL)$.
Let $\calB_{\bL}$ be the refinement of the base $\calB$
for $(X_K)_{\et}$ whose objects consist of $(Y=\Spa(B,B^+)\ra X_{k'})\in\calB$
such that $\calH(\bL)(Y_K)$ is a finite free $B_K$-module.

For $(Y=\Spa(B,B^+)\ra X_{k'})\in\calB_{\bL}$, 
$\calH(\bL)(Y_K)$ is a $B_K$-representation
 of $\Gamma_{k'}:=\Gal(K/k')$ in the sense of Definition~\ref{defn:representation}. 
Thus Proposition~\ref{prop:Sen endom} and Definition~\ref{defn:base chage of Sen endom}
equip $\calH(\bL)(Y_K)$ with the $B_K$-linear endomorphism
\[
 \phi_{\calH(\bL)(Y_K)}\colon \calH(\bL)(Y_K)\ra \calH(\bL)(Y_K).
\]

\begin{lem-defn}
The assignment of endomorphisms 
\[
 \calB_{\bL}\ni(Y=\Spa(B,B^+)\ra X_{k'})\lmt
\phi_{\calH(\bL)(Y_K)}\in \End_{B_K}\calH(\bL)(Y_K)
\]
defines an endomorphism $\phi_{\bL}$
of the vector bundle $\calH_{\bL}$ on $(X_K)_{\et}$.
We call $\phi_{\bL}$ the \emph{arithmetic Sen endomorphism} of $\bL$.
\end{lem-defn}

\begin{proof}
We need to check the compatibility of $\phi_{\bL,Y_K}$
via the pullback $Y''_K\ra Y_K$
for $(Y=\Spa(B,B^+)\ra X_{k'}), (Y''=\Spa(B'',B''^+)\ra X_{k''})\in\calB_{\bL}$.
For this it suffices to prove that 
\[
 \calH(\bL)(Y_K)_{\fin}\otimes_{B_{\infty}}B''_{\infty}
\cong \calH(\bL)(Y''_K)_{\fin}
\]
as $B''_{\infty}$-representation of $\Gal(k_\infty/k'')$,
where $B_\infty$ and $B''_\infty$ 
are defined as in Subsection~\ref{subsection:setup}.

Since $\calH(\bL)$ is a vector bundle on $X_K$, we have the natural isomorphisms
\begin{align*}
  \bigl(\calH(\bL)(Y_K)_{\fin}\otimes_{B_{\infty}}B''_{\infty}\bigr)
\otimes_{B''_{\infty}}B''_K
&\cong \bigl(\calH(\bL)(Y_K)_{\fin}\otimes_{B_{\infty}}B_K\bigr)
\otimes_{B_K}B''_K\\
&\cong \calH(\bL)(Y_K)\otimes_{B_K}B''_K
\cong \calH(\bL)(Y''_K).
\end{align*}
On the other hand, 
we see from definition
$\calH(\bL)(Y_K)_{\fin}\otimes_{B_{\infty}}B''_{\infty}\subset \calH(\bL)(Y''_K)_{\fin}$ .
Hence the lemma follows from the faithful flatness of $B''_{\infty}\ra B''_K$.
\end{proof}

\begin{prop}\label{prop:commutativity of Higgs and Sen}
The following diagram commutes:
\[
 \xymatrix{
\calH(\bL)\ar[r]^-{\vartheta_{\bL}}\ar[d]_{\phi_{\bL}}
&\calH(\bL)\otimes\Omega_{X/k}^1(-1)
\ar[d]^{\phi_{\bL}\otimes\id-\id\otimes\id}\\
\calH(\bL)\ar[r]^-{\vartheta_{\bL}}
&\calH(\bL)\otimes\Omega_{X/k}^1(-1).
}
\]
In particular, the endomorphisms
$\phi_{\bL}\otimes \id - i(\id\otimes\id)$
on $\calH(\bL)\otimes\Omega_{X/k}^i(-i)$
give rise to an endomorphism on the complex of $\calO_{X_K}$-modules on $X_K$
\[
 \calH(\bL)\stackrel{\vartheta_{\bL}}{\lra}
\calH(\bL)\otimes\Omega_{X/k}^1(-1)
\stackrel{\vartheta_{\bL}}{\lra}
\calH(\bL)\otimes\Omega_{X/k}^2(-2)\lra \cdots
\]
induced by the Higgs field.
\end{prop}

\begin{proof}
It is enough to check the commutativity of the diagram
evaluated at $Y_K$ for each
$(Y=\Spa(B,B^+)\ra X_{k'})\in\calB_{\bL}$.
In this setting, we can use the identification
\[
 \bigl(\calH(\bL)(Y_K),\vartheta_{\bL}(Y_K),\phi_{\bL}(Y_K)\bigr)
\cong \bigl(M_K(Y),\vartheta_{M_K(Y)},\phi_{M_K(Y)}\bigr).
\]
So it suffices to show the commutativity of the diagram
\[
\xymatrix{
M_K(Y)\ar[r]^-{\vartheta_{M_K(Y)}}\ar[d]_{\phi_{M_K(Y)}}
&M_K(Y)\otimes_B\Omega_{B/k'}^1(-1)
\ar[d]^{\phi_{M_K(Y)}\otimes\id-\id\otimes\id}\\
M_K(Y)\ar[r]^-{\vartheta_{M_K(Y)}}
&M_K(Y)\otimes_B\Omega_{B/k'}^1(-1).
}
\]
Moreover, since $M(Y)\otimes_{B_{k_{j_0}}}B_K=M_K(Y)$, 
we only need to check the commutativity on $M(Y)\subset M_K(Y)$.

We use the notation in (\ref{eq:Higgs field of M_K(Y)}).
Then we have
\begin{align*}
 \vartheta_{M_K(Y)}\circ \phi_{M_K(Y)}
&=\sum_{i=1}^n (\vartheta_i\circ \phi_{M_K(Y)})\otimes \frac{dT_i}{T_i}\otimes t^{-1},\quad \text{and}\\
 (\phi_{M_K(Y)}\otimes\id-\id\otimes\id)\circ \vartheta_{M_K(Y)}
&=\sum_{i=1}^n (\phi_{M_K(Y)}\circ\vartheta_i-\vartheta_i)\otimes \frac{dT_i}{T_i}\otimes t^{-1}.
\end{align*}

Thus we need to show that $[\phi_{M_K(Y)},\vartheta_i]=\vartheta_i$
for each $i$.
To see this, take a topological generator $\gamma_i$
of the $i$-th component of $\Gamma_{\geom}\cong \Z_p(1)^n$.
Let $\rho_{\geom}$ denote
the action of $\Gamma_{\geom}$ on $M_K(Y)$ 
and write $\log \rho_{\geom}=\sum_{i=1}^n \vartheta_i\otimes \chi_i\otimes t^{-1}$ as before.
Since
 $\gamma \gamma_i \gamma^{-1}=\gamma_i^{\chi(\gamma)}$
for $\gamma\in \Gamma_k$,
we have
\[
 \gamma (\log \rho_{\geom}(\gamma_i))\gamma^{-1}
=\log \rho_{\geom}(\gamma\gamma_i\gamma^{-1})
=\chi(\gamma)\log \rho_{\geom}(\gamma_i).
\]
Hence $\gamma\vartheta_i=\chi(\gamma)\vartheta_i \gamma$
for $\gamma\in \Gamma_k$.

For $m\in M(Y)$, we compute
\begin{align*}
\phi_{M_K(Y)} \vartheta_i m
&=\lim_{j\to \infty} \frac{1}{\log \chi(\gamma)}
\frac{\gamma^{p^j}\vartheta_i m-\vartheta_i m}{p^j} \\
&=\lim_{j\to \infty} \frac{1}{\log \chi(\gamma)}
\frac{(\chi(\gamma)^{p^j}-1)\vartheta_i\gamma^{p^j} m
+\vartheta_i (\gamma^{p^j}m-m)}{p^j} \\
&=\vartheta_i m + \vartheta_i \phi_{M_K(Y)} m.
\end{align*}
Hence $[\phi_{M_K(Y)},\vartheta_i]=\vartheta_i$.
\end{proof}

\begin{rem}
Brinon generalized Sen's theory to the case 
of $p$-adic fields with imperfect residue fields
in \cite{Brinon}.
Analogues of $\phi_{\bL}$ and $\vartheta_i$ 
have already appeared in his work.
\end{rem}

We discuss properties of the arithmetic Sen endomorphism 
along the lines of 
Theorem~\ref{thm:Simpson by LZ} (i.e., \cite[Theorem 2.1]{LZ}).

\begin{thm}\hfill
\begin{enumerate}
 \item 
There are canonical isomorphisms
\[
 (\calH(\bL_1\otimes\bL_2),\vartheta_{\bL_1\otimes\bL_2},
\phi_{\bL_1\otimes\bL_2})\cong
(\calH(\bL_1)\otimes\calH(\bL_2),\vartheta_{\bL_1}\otimes\id+
\id\otimes\vartheta_{\bL_2},\phi_{\bL_1}\otimes\id+\id\otimes\phi_{\bL_2})
\]
and
\[
 (\calH(\bL^\vee),\vartheta_{(\bL^\vee)},\phi_{(\bL^\vee)})
\cong(\calH(\bL)^\vee,(\vartheta_{\bL})^\vee,(\phi_{\bL})^\vee).
\]
\item
Let $f\colon Y\ra X$ be a morphism between smooth rigid analytic varieties over $k$ and $\bL$ be a $\Q_p$-local system on $X_{\et}$.
Then there is a canonical isomorphism
\[
 f^\ast(\calH(\bL),\vartheta_{\bL},\phi_{\bL})\cong
(f^\ast\calH(\bL),\vartheta_{f^\ast\bL},\phi_{f^\ast\bL}).
\]
\end{enumerate}
\end{thm}

\begin{proof}
Part (i) follows from \cite[Theorem 2.1(iv)]{LZ} and 
Lemma~\ref{lem:properties of Sen endom}.
Part (ii) follows from  
\cite[Theorem 2.1(iii)]{LZ}
and Proposition~\ref{prop:Sen endom} (functoriality of $\phi_W$).
\end{proof}

By construction, we also have the following for the case of points.

\begin{prop}
If $X$ is a point, then $\phi_{\bL}$ coincides with the Sen endomorphism
attached to the Galois representation $\bL$.
\end{prop}

For pushforwards, we have the following theorem 
(the notation is explained after the statement).

\begin{thm}\label{thm:pushforward of Sen endomorphism}
Let $f \colon X\ra Y$ be a smooth proper morphism 
between smooth rigid analytic varieties over $k$
and let $\bL$ be a $\Z_p$-local system on $X_{\et}$.
Assume that $R^if_\ast \bL$ is a $\Z_p$-local system on $Y_{\et}$
for every $i$.
Then we have
\[
 (\calH(R^if_{\ast}\bL), \vartheta_{R^if_{\ast}\bL})
\cong R^if_{\Higgs,\ast}
(\calH(\bL)\otimes \Omega^\bullet_{X/Y}(-\bullet),\bar{\vartheta_{\bL}}).
\]
Moreover, under this isomorphism, we have
\[
 \phi_{R^if_{\ast}\bL}
=R^if_{K,\et,\ast}(\phi_\bL\otimes \id - \bullet (\id\otimes \id)).
\]
\end{thm}

Let us explain the notation in the theorem.
Recall 
the complex of $\calO_{X_K}$-modules 
\[
 \calH(\bL)\stackrel{\vartheta_{\bL}}{\lra}
\calH(\bL)\otimes\Omega_{X/k}^1(-1)
\stackrel{\vartheta_{\bL}}{\lra}
\calH(\bL)\otimes\Omega_{X/k}^2(-2)\lra \cdots.
\]
This has an $\calO_X$-linear endomorphism
$\phi_\bL\otimes\id - \bullet(\id\otimes \id)$
by Proposition~\ref{prop:commutativity of Higgs and Sen}.
The complex yields
a  complex of $\calO_{X_K}$-modules 
\[
 \calH(\bL)\stackrel{\bar{\vartheta}_{\bL}}{\lra}
\calH(\bL)\otimes\Omega_{X/Y}^1(-1)
\stackrel{\bar{\vartheta}_{\bL}}{\lra}
\calH(\bL)\otimes\Omega_{X/Y}^2(-2)\lra \cdots
\]
by composing with the projection $\Omega^i_{X/k}\ra \Omega^i_{X/Y}$.
The new complex has an induced $\calO_X$-linear endomorphism,
which we still denote by 
$\phi_\bL\otimes\id - \bullet(\id\otimes \id)$.

We denote by $f_K\colon X_K\ra Y_K$ the base change of $f$.
Then $R^if_{\Higgs,\ast}$
is the $i$-th derived pushforward of the complex with the Higgs field.
In particular, 
$R^if_{\Higgs,\ast}
(\calH(\bL)\otimes \Omega^\bullet_{X/Y}(-\bullet),\bar{\vartheta_{\bL}})$
is the $\calO_{X_K,\et}$-module
$R^if_{K,\et,\ast}(\calH(\bL)\otimes \Omega^\bullet_{X/Y}(-\bullet))$
together with a Higgs field.

\begin{proof}
The first part is \cite[Theorem 2.1(v)]{LZ} (see Theorem~\ref{thm:Simpson by LZ}).
So we will prove the statement on arithmetic Sen endomorphisms.

Since the statement is local on $Y$, we may assume that $Y$ is an affinoid
$\Spa (A,A^+)$ and that
$\calH(R^if_{\ast}\bL)$ is a globally free vector bundle on $Y_K$.
So $\calH(R^if_{\ast}\bL)$ is associated to 
a finite free $A_K$-module (say $V$).
Then $V$ is an $A_K$-representation of $\Gamma_k$ and 
the endomorphism $\phi_{R^if_{\ast}\bL}$ is associated to $\phi_V$.

Since $X$ is quasi-compact, there exists a finite affinoid open cover
$X=\bigcup_{i\in I}U^{(i)}$
with $U^{(i)}=\Spa (B^{(i)}, B^{(i),+})$
such that $\calH(\bL)|_{U^{(i)}_K}$ is a globally finite free
vector bundle for each $i$.
So $\calH(\bL)|_{U^{(i)}_K}$ gives rise to
a $B_K^{(i)}$-representation of $\Gamma_k$ and 
the latter is defined over $B_{k_m}^{(i)}$ for a sufficiently large $m$
(cf.~the proof of Theorem~\ref{thm:Sen decompletion}).
Since the same holds for $\calH(\bL)|_{U^{(i)}_K\cap U^{(j)}_K}$,
there exists a large integer $m$ such that 
 the complex of $\calO_{X_K}$-modules
\[
 \calH(\bL)\stackrel{\bar{\vartheta}_{\bL}}{\lra}
\calH(\bL)\otimes\Omega_{X/Y}^1(-1)
\stackrel{\bar{\vartheta}_{\bL}}{\lra}
\calH(\bL)\otimes\Omega_{X/Y}^2(-2)\lra \cdots
\]
with the $\Gamma_k$-action and the endomorphism 
$\phi_\bL\otimes\id - \bullet(\id\otimes \id)$
descends to a complex of $\calO_{X_{k_m}}$-modules
\[
 \calH(\bL)_{k_m}\stackrel{\bar{\vartheta}_{\bL}}{\lra}
\calH(\bL)_{k_m}\otimes\Omega_{X/Y}^1(-1)
\stackrel{\bar{\vartheta}_{\bL}}{\lra}
\calH(\bL)_{k_m}\otimes\Omega_{X/Y}^2(-2)\lra \cdots
 \]
on $X_{k_m}$ equipped with a $\Gamma$-action and an endomorphism
$\phi_\bL\otimes\id - \bullet(\id\otimes \id)$
such that
$\calH(\bL)_{k_m}|_{U^{(i)}_{k_m}}$ is 
a globally finite free vector bundle for each $i$.
 We denote by $\calF^\bullet$ the complex on $X_{k_m}$ 
and by $\phi^\bullet$ the descended endomorphism.

Let $f_{k_m}\colon X_{k_m}\ra Y_{k_m}$ denote the base change of $f$.
Set
\[
 \calH_{Y,k_m}:=R^if_{k_m,\et}\calF^\bullet.
\]
Since $f_{k_m}$ is proper, this is a coherent $\calO_{Y_{k_m}}$-module by
Kiehl's finiteness theorem\footnote{For a coherent $\calO_{k_m}$-module $\calF$, 
we have $(R^i f_{k_m}\calF)_{\et}=R^i f_{k_m,\et}\calF_{\et}$ 
(\cite[Proposition 9.2]{Scholze}).
So we simply write $\calF$ for the sheaf $\calF_{\et}$ on $X_{k_m,\et}$.}.
We have 
\[
  \calH_{Y,k_m}|_{Y_K}
= (R^if_{k_m,\et}\calF^\bullet)|_{Y_K}
=R^i f_{K,\et}(\calF^\bullet|_{X_K})
=R^i f_{K,\et}(\calH(\bL)\otimes \Omega_{X/Y}^\bullet (-\bullet)),
\]
and this is isomorphic to $\calH(R^if_{\ast}\bL)$
by the first assertion.
Thus (after increasing $m$) $\calH_{Y,k_m}$ is globally finite free
and associated to a finite free $A_{k_m}$-module (say $V_{k_m}$) 
with a $\Gamma_k$-action satisfying $V_{k_m}\otimes_{A_{k_m}}A_K=V$.
By construction, $V_{k_m}$ is contained in $V_{\fin}$
and the $A_K$-linear endomorphism $\phi_V$ on $V$ 
is uniquely characterized by the following property:
for each $v\in V_{k_m}$, there exists 
an open subgroup $\Gamma_k'\subset \Gamma_k$
such that 
\[
 \exp(\log \chi(\gamma)\phi_V)v=\gamma v
\]
for all $\gamma\in \Gamma_k'$.

We will show that $R^if_{K,\et,\ast}(\phi_\bL\otimes \id - \bullet (\id\otimes \id))$ defines an $A_K$-linear endomorphism on $V$ with the same property.
To see this, we compute $V_{k_m}$ via the \v{C}ech-to-derived functor spectral sequence.
Note that
\[
 V_{k_m}
=\Gamma(Y_{k_m,\et},\calH_{Y,k_m})
=R^i\Gamma(X_{k_m,\et},\calF^\bullet)
\]
by definition.

Let us briefly recall the \v{C}ech-to-derived functor spectral sequence.
Set $\calU:=\{U_{k_m}^{(i)}\}_{i\in I}$. For $i_0,\ldots,i_a\in I$, 
we denote by $U_{k_m}^{(i_0\cdots i_a)}$ the affinoid open
$U_{k_m}^{(i_0)}\cap\cdots\cap U_{k_m}^{(i_a)}$.
Consider the \v{C}ech double complex 
$\check{C}^\bullet(\calU,\calF^\bullet)$
associated to the complex $\calF^\bullet$;
this is defined by
\[
 \check{C}^a(\calU,\calF^b)
:=\prod_{i_0,\ldots, i_a\in I}\calF^b(U_{k_m}^{(i_0\cdots i_a)}).
\]
Let $\underline{H}^b$ be the $b$-th right derived functor
of the forgetful functor from the category of abelian sheaves on
$X_{k_m,\et}$ to the category of abelian presheaves on
$X_{k_m,\et}$;
for an abelian sheaf $\calG$, $\underline{H}^b(\calG)$
associates to $(U\ra X_{k_m})$ 
the abelian group $H^b(U,\calG)$.
Then the \v{C}ech-to-derived functor spectral sequence is a
spectral sequence with
\[
 E_2^{a,b}=H^a(\Tot(\check{C}^\bullet(\calU,\underline{H}^b(\calF^\bullet))))
\]
converging to $R^{a+b}\Gamma(X_{k_m,\et},\calF^\bullet)$.
Moreover, this is functorial in $\calF^\bullet$.

In our case, $\calF^\bullet$ consists of coherent $\calO_{k_m}$-modules
and $U_{k_m}^{(i_0\cdots i_a)}$ are all affinoid.
So $\underline{H}^b(\calF^c)(U_{k_m}^{(i_0\cdots i_a)})=0$
for each $b>0$ and any $a$ and $c$ by Kiehl's theorem.
Thus the spectral sequence yields an isomorphism
\[
 H^i(\Tot(\check{C}^\bullet(\calU,\calF^\bullet)))
\stackrel{\cong}{\lra}
R^i\Gamma(X_{k_m,\et},\calF^\bullet)=V_{k_m}.
 \]
Moreover, this isomorphism is 
$\Gamma_k$-equivariant as the construction is functorial in $\calF^\bullet$.

Let us unwind the definition of $\check{C}^a(\calU,\calF^b)$:
\[
 \check{C}^a(\calU,\calF^b)
=\prod_{i_0,\ldots, i_a\in I}
\Gamma(U_{k_m}^{(i_0\cdots i_a)}, \calH(\bL)_{k_m}\otimes\Omega_{X/Y}^b(-b)).
\]
Set 
\[
W^{(i_0\cdots i_a),b}:=\Gamma(U_{K}^{(i_0\cdots i_a)}, \calH(\bL)\otimes\Omega_{X/Y}^b(-b))
\]
and 
\[
 W_{k_m}^{(i_0\cdots i_a),b}:=\Gamma(U_{k_m}^{(i_0\cdots i_a)}, \calH(\bL)_{k_m}\otimes\Omega_{X/Y}^b(-b)).
\]
They have a natural $\Gamma_k$-action, and 
$W_{k_m}^{(i_0\cdots i_a),b}$ is contained in $(W^{(i_0\cdots i_a),b})_{\fin}$.
In particular, the restriction of $\phi_{W^{(i_0\cdots i_a),b}}$ to
$W_{k_m}^{(i_0\cdots i_a),b}$ satisfies the following property:
for each $w\in W_{k_m}^{(i_0\cdots i_a),b}$, there exists 
an open subgroup $\Gamma_k'\subset \Gamma_k$
such that 
\[
 \exp(\log \chi(\gamma)\phi_{W^{(i_0\cdots i_a),b}})w=\gamma w
\]
for all $\gamma\in \Gamma_k'$.

It follows from our construction that 
under the isomorphism
$ H^i(\Tot(\check{C}^\bullet(\calU,\calF^\bullet)))
\cong
R^i\Gamma(X_{k_m,\et},\calF^\bullet)=V_{k_m}$,
the endomorphism 
$R^if_{K,\et,\ast}(\phi_\bL\otimes \id - \bullet (\id\otimes \id))|_{V_{k_m}}
=R^if_{k_m,\et,\ast}\phi^\bullet$
corresponds to
\[
 H^i(\Tot(\prod_{i_0,\ldots, i_a\in I}\phi_{W^{(i_0\cdots i_a),b}})).
\]
Since differentials in the complex $\Tot(\check{C}^\bullet(\calU,\calF^\bullet))$ are all $\Gamma_k$-equivariant, we see that
$H^i(\Tot(\prod_{i_0,\ldots, i_a\in I}\phi_{W^{(i_0\cdots i_a),b}}))$
satisfies the above-mentioned characterizing
property of $\phi_V$.
This completes the proof.
\end{proof}

\section{Constancy of generalized Hodge-Tate weights}
\label{section:constancy}
In this section, we prove the multiset of eigenvalues of $\phi_{\bL}$
is constant on $X_K$
(Theorem~\ref{thm:constancy of eigenvalues of arithmetic Sen operator}).
For this
we give a description of $\phi_{\bL}$ as the residue of a formal connection
in Subsection~\ref{subsection:review of gemetric Riemann-Hilbert}.
Then the constancy is proved by the theory 
of formal connections 
developed in Subsection~\ref{subsection:formal connections}.

\subsection{The decompletion of the geometric Riemann-Hilbert correspondence}
\label{subsection:review of gemetric Riemann-Hilbert}
 We review the geometric Riemann-Hilbert correspondence by Liu and Zhu
and discuss its decompletion.

Keep the notation in Section~\ref{section:arithmetic Sen endomorphism}.
Let $\bL$ be a $\Q_p$-local system on $X_{\et}$ of rank $r$.
Following \cite{LZ}, we define
\[
 \RH(\bL)=\nu'_\ast\bigl(\hat{\bL}\otimes_{\hat{\Q}_p}\OB_{\dR}\bigr).
\]

In order to state their theorem,
let us recall a ringed space $\calX$ introduced in \cite[\S 3.1]{LZ}.
Let $L_{\dR}^+$ denote the de Rham period ring 
$\bB_{\dR}^+(K,\calO_K)$ as before 
(\cite{LZ} uses $B_{\dR}^+$ but we prefer to use $L_{\dR}^+$).
Define a sheaf $\calO_X\hat{\otimes}(L_{\dR}^+/t^i)$ on $X_{K,\et}$
by assigning
\[
 (Y=\Spa(B,B+)\ra X_{k'})\in\calB
\longmapsto B\hat{\otimes}_{k'}(L_{\dR}^+/t^i).
\]
This  defines a sheaf by the Tate acyclicity theorem.
We also set
\[
 \calO_X\hat{\otimes}L_{\dR}^+=\varprojlim_i\calO_X\hat{\otimes}(L_{\dR}^+/t^i)
\]
and
\[
 \calO_X\hat{\otimes}L_{\dR}=(\calO_X\hat{\otimes}L_{\dR}^+)[t^{-1}].
\]
We denote the ringed space $(X_K,\calO_X\hat{\otimes}L_{\dR})$
by $\calX$.
We have a natural base change functor $\calE\mapsto \calE\hat{\otimes}L_{\dR}$
from the category of vector bundles on $X$ to the category of vector bundles on $\calX$. We set
\[
 \Omega_{\calX/L_{\dR}}^1=\Omega_{X/k}^1\hat{\otimes} L_{\dR}.
\]

\begin{thm}[{\cite[Theorem 3.8]{LZ}}]\label{thm:LZ Theorem 3.8}
\hfill
\begin{enumerate}
 \item $\RH(\bL)$ is a filtered vector bundle on $\calX$
of rank $r$
equipped with an integrable connection
\[
\RH(\bL)\stackrel{\nabla_{\bL}}{\lra}
\RH(\bL)\otimes_{\calO_{\calX}}\Omega_{\calX/L_{\dR}}
\]
that satisfies the Griffiths transversality.
Moreover, $\Gal(K/k)$ acts on $\RH(\bL)$ semilinearly,  
and the action preserves the filtration and commutes with $\nabla_{\bL}$.
 \item
There is a canonical isomorphism
\[
 (\gr^0\RH(\bL),\gr^0(\nabla_{\bL}))
\cong (\calH(\bL),\vartheta_{\bL}).
\]
\end{enumerate}
\end{thm}

We want to consider a decompletion of $\RH(\bL)$.
Here we only develop an ad hoc local theory that is sufficient for our purpose.

Take $(Y=\Spa(B,B^+)\ra X_{k'})\in \calB_{\bL}$ and 
consider $\Fil^0\RH(\bL)(Y_K)$.
This is a $B\hat{\otimes}_{k'}L_{\dR}^+$-representation of $\Gal(K/k')$, 
and $\RH(\bL)(Y_K)=\bigl( \Fil^0\RH(\bL)(Y_K)\bigr)[t^{-1}]$.
Moreover, since $\gr^0\RH(\bL)(Y_K)$ is a finite free $B_K$-module
by the definition of $\calB_{\bL}$, 
the $B\hat{\otimes}_{k'}L_{\dR}^+$-module $\Fil^0\RH(\bL)(Y_K)$ is also finite free.

Definition~\ref{defn:de Rham connection} yields the $B_{\infty}((t))$-module $\RH(\bL)(Y_K)_{\fin}$ and 
the $B_{\infty}$-linear endomorphism
\[
 \phi_{\dR,\RH(\bL)(Y_K)_{\fin}}\colon \RH(\bL)(Y_K)_{\fin}\ra \RH(\bL)(Y_K)_{\fin}.
\]
For simplicity, we denote $\phi_{\dR,\RH(\bL)(Y_K)_{\fin}}$ by $\phi_{\dR,\bL,Y_K}$.
It satisfies
\[
 \phi_{\dR,\bL,Y_K}(\alpha m)=t\partial_t(\alpha)m+\alpha\phi_{\dR,\bL,Y_K}(m)
\]
for every $\alpha\in B_{\infty}((t))$ and $m\in \RH(\bL)(Y_K)_{\fin}$.
Note that $\nabla_{\bL}$ is $\Gal(K/k)$-equivariant.
Hence under the identification
\[
\RH(\bL)\otimes_{\calO_{\calX}}\Omega_{\calX/L_{\dR}}^1\cong
\RH(\bL)\otimes_{\calO_X}\Omega^1_{X/k},
\]
we have
\[
 \nabla_{\bL,Y_K} \bigl(\RH(\bL)(Y_K)_{\fin}\bigr) \subset 
(\RH(\bL)(Y_K)\otimes \Omega^1_{B/k'})_{\fin}
=\RH(\bL)(Y_K)_{\fin}\otimes \Omega^1_{B/k'}.
\]

\begin{prop}\label{prop:commutativity of de Rham Sen and connection}
The following diagram commutes:
\[
 \xymatrix{
\RH(\bL)(Y_K)_{\fin}\ar[r]^-{\nabla_{\bL,Y_K}}\ar[d]_{\phi_{\dR,\bL,Y_K}}
&\RH(\bL)(Y_K)_{\fin}\otimes_{B}\Omega^1_{B/k'}
\ar[d]^{\phi_{\dR,\bL,Y_K}\otimes \id}\\
\RH(\bL)(Y_K)_{\fin}\ar[r]^-{\nabla_{\bL,Y_K}}
&\RH(\bL)(Y_K)_{\fin}\otimes_{B}\Omega^1_{B/k'}.
}
\]
Moreover, we have
\[
 \Res_{\Fil^0 \RH(\bL)(Y_K)_{\fin}}\phi_{\dR,\bL,Y_K}
=\phi_{\bL, Y_K}.
\]
\end{prop}

\begin{proof}
The commutativity of the diagram follows from the fact that
$\nabla_{\bL}$ is $\Gal(K/k)$-equivariant.
The second assertion is 
a consequence of Theorem~\ref{thm:LZ Theorem 3.8}(ii)
(cf.~Definition~\ref{defn:de Rham connection}).
\end{proof}

\begin{rem}
In \cite{AB}, Andreatta and Brinon 
developed a Fontaine-type decompletion theory in the relative setting.
Roughly speaking, they associated to a local system on $X$
a formal connection over the pro-\'etale cover $\widetilde{X}_{K,\infty}$
over $X_K$ when $X$ is an affine scheme admitting a toric chart.
\end{rem}

\subsection{Theory of formal connections}
\label{subsection:formal connections}
To study $\phi_{\dR,\bL, Y_K}$ in the previous subsection,
we develop a theory of formal connections.
We work on the following general setting:
let $R$ be an integral domain of characteristic $0$
(e.g.~ $R=B_{\infty}$ in the previous subsection)
and fix an algebraic closure of the fraction field of $R$.
Consider the ring of Laurent series $R((t))$ and
define the $R$-linear derivation 
$d_0\colon R((t))\ra R((t))$
by
\[
 d_0\biggl(\sum_{j\in\Z}a_jt^j\biggr)=\sum_{j\in\Z}ja_jt^{j-1}.
\]
Let $M$ be a finite free $R((t))$-module of rank $r$ and
let $D_0\colon M\ra M$ be an $R$-linear map which satisfies the Leibniz rule
\[
 D_0(\alpha m)=\alpha D_0(m)+d_0(\alpha)m\quad (\alpha\in R((t)), m\in M).
\]

\begin{defn}
A \textit{$tD_0$-stable lattice} of $M$ is a finite free $R[[t]]$-submodule $\Lambda$ of $M$
that satisfies 
\[
 \Lambda\otimes_{R[[t]]}R((t))=M\quad \text{and}\quad
tD_0(\Lambda)\subset \Lambda.
\] 
\end{defn}

For a $tD_0$-stable lattice $\Lambda$ of $M$, we have
$tD_0(t\Lambda)\subset t\Lambda$ by the Leibniz rule.
Thus $tD_0\colon \Lambda\ra \Lambda$ induces an $R$-linear endomorphism
on $\Lambda/t\Lambda$. We denote this endomorphism by $\Res_\Lambda D_0$.
Since $\Lambda/t\Lambda$ is a finite free $R$-module of rank $r$,
the endomorphism $\Res_\Lambda D_0$ has $r$ eigenvalues (counted with multiplicity) in the algebraic closure of the fraction field of $R$.

The following is known for $tD_0$-stable lattices.
\begin{thm}\label{thm:diff eq}
Assume that $R$ is an algebraically closed field.
\begin{enumerate}
 \item 
There exists a finite subset $\calA$ of $R$ such that
the submodule
\[
 \Lambda_{\calA}:=\bigoplus_{\alpha\in\calA}\Ker (tD_0-\alpha)^r\otimes_RR[[t]]
\]
is a $tD_0$-stable lattice of $M$.
In particular, the eigenvalues of $\Res_{\Lambda_{\calA}}D_0$
lie in $\calA$.
 \item For any $tD_0$-stable lattices $\Lambda$ and $\Lambda'$ of $M$,
the eigenvalues of $\Res_{\Lambda}D_0$ and those of $\Res_{\Lambda'}D_0$
differ by integers. Namely,
for each eigenvalue $\alpha$ of $\Res_{\Lambda}D_0$,
there exists an eigenvalue $\alpha'$ of $\Res_{\Lambda'}D_0$
such that $\alpha-\alpha'\in \Z$.
\end{enumerate}
 \end{thm}
See \cite[III.8 and V. Lemma 2.4]{DGS} and \cite[Proposition 3.2.2]{Andre-Baldassarri} for details.

We now turn to the following multivariable situation:
Let $R$ be an integral domain of characteristic 0 as before.
Suppose that $R$ is equipped with pairwise commuting derivations
$d_1,\ldots,d_n$;
this means that for each $i=1,\ldots, n$, the map $d_i\colon R\ra R$
is additive and satisfies the Leibniz rule
\[
 d_i(ab)=d_i(a)b+ad_i(b)\quad (a,b\in R),
\]
and $d_i\circ d_j=d_j\circ d_i$ for each $i$ and $j$.
Since $R$ is an integral domain of characteristic 0, the derivations $d_1,\ldots,d_n$ extend uniquely over the algebraic closure of the fraction field of $R$.

Consider the ring of Laurent series $R((t))$ and
define the $R$-linear derivation $d_0\colon R((t))\ra R((t))$
by
\[
 d_0\biggl(\sum_{j\in\Z}a_jt^j\biggr)=\sum_{j\in\Z}ja_jt^{j-1}.
\]
For each $i=1,\ldots,n$, we extend $d_i\colon R\ra R$ to an additive map
$d_i\colon R((t))\ra R((t))$ by
\[
 d_i\biggl(\sum_{j\in\Z}a_jt^j\biggr)=\sum_{j\in\Z}d_i(a_j)t^j.
\]
Then endomorphisms $d_0,d_1,\ldots,d_n$ on $R((t))$ commute with each other.
Moreover, $d_1,\ldots,d_n$ commute with $td_0$.

Let $M$ be a finite free $R((t))$-module of rank $r$ 
together with pairwise commuting additive endomorphisms
$D_0, D_1,\ldots, D_n\colon M\ra M$ satisfying the Leibniz rules
\[
 D_i(\alpha m)=\alpha D_i(m)+d_i(\alpha)m\quad (\alpha\in R((t)), m\in M, 0\leq i\leq n).
\]
Note that $D_0$ is $R$-linear and $D_1,\ldots,D_n$ commute with $tD_0$.

The following proposition is the key to the constancy of generalized Hodge-Tate weights.

\begin{prop}\label{prop:constancy of residue}
With the notation as above, let $\Lambda$ be a $tD_0$-stable lattice of $M$.
Then each eigenvalue $\alpha$ of 
$\Res_{\Lambda}D_0$ in the algebraic closure of the fraction field of $R$
satisfies
\[
d_1(\alpha)=\cdots=d_n(\alpha)=0. 
\]
\end{prop}

\begin{proof}
By extending scalars from $R$ to the algebraic closure of its fraction field, we may assume that $R$ is an algebraically closed field.
By Theorem~\ref{thm:diff eq} (i),
there exists a finite subset $\calA$ of $R$ such that
the submodule
\[
 \Lambda_{\calA}:=\bigoplus_{\alpha\in\calA}\Ker (tD_0-\alpha)^r\otimes_RR[[t]]
\]
is a $tD_0$-stable lattice of $M$.

By Theorem~\ref{thm:diff eq} (ii), 
the eigenvalues of $\Res_{\Lambda}D_0$ and those of $\Res_{\Lambda_{\calA}}D_0$
differ by integers.
Since every integer $a$ satisfies $d_1(a)=\cdots=d_n(a)=0$,
it suffices to treat the case where $\Lambda=\Lambda_{\calA}$.

\begin{lem}\label{lem:ker lattice is D-stable}
 The finite free $R[[t]]$-submodule $\Lambda_{\calA}$ is
stable under $D_1,\ldots,D_n$.
\end{lem}

Note that
Lemma~\ref{lem:ker lattice is D-stable} says
that the connection $(\Lambda_{\calA}, D_0,\ldots,D_n)$
is regular singular along $t=0$.
In this case, Proposition~\ref{prop:constancy of residue} is easy to prove.
In fact, this is an algebraic analogue
of the following fact:
let $X$ be  the complex affine space $\A^n_{\C}$ and $D$ the divisor
$\{0\}\times \A^{n-1}_{\C}$.
Consider a vector bundle $\Lambda$ on $X$ 
and an integrable connection $\Gamma$ on $\Lambda|_{X\setminus D}$
that admits a logarithmic pole along $D$.
Let $T$ be the monodromy transformation of $\Lambda|_{X\setminus D}$
defined by the positive generator of $\pi_1(X\setminus D)=\Z$.
Then $T$ extended to an automorphism $\widetilde{T}$ of $\Lambda$
and satisfies
\[
 \widetilde{T}|_D=\exp (-2\pi i \Res_D \Lambda).
\]
See \cite[Proposition 3.11]{Deligne-diffeq}.

\begin{proof}[Proof of Lemma~\ref{lem:ker lattice is D-stable}]
This is \cite[Lemma 3.3.2]{Andre-Baldassarri}.
For the convenience of the reader, we reproduce the proof here.
Fix $1\leq i\leq n$ and $\alpha\in\calA$.
It is enough to show that for each $0\leq j\leq r$, 
\[
 D_i\Ker (tD_0-\alpha)^j\subset \Ker (tD_0-\alpha)^{j+1}.
\]

We prove this inclusion by induction on $j$. The assertion is trivial when $j=0$.
Assume $j>0$ and take $m\in \Ker (tD_0-\alpha)^j$.
Then $(tD_0-\alpha)m\in \Ker (tD_0-\alpha)^{j-1}$, and thus
$D_i (tD_0-\alpha)m\in \Ker (tD_0-\alpha)^j$
by the induction hypothesis.
We need to show $(tD_0-\alpha)^{j+1}D_im=0$.
Since $D_i$ commutes with $tD_0$ and satisfies $D_i(\alpha m)=\alpha D_i(m)+d_i(\alpha)m$, we have
\[
 (tD_0-\alpha)D_im=D_i(tD_0-\alpha)m+d_i(\alpha)m.
\]
Therefore 
\begin{align*}
  (tD_0-\alpha)^{j+1}D_im
&= (tD_0-\alpha)^j  (tD_0-\alpha)D_im \\
&=(tD_0-\alpha)^j  D_i(tD_0-\alpha)m+(tD_0-\alpha)^j d_i(\alpha)m\\
&=(tD_0-\alpha)^j  D_i(tD_0-\alpha)m+d_i(\alpha)(tD_0-\alpha)^j m.
\end{align*}
For the third equality, note that $d_i(\alpha)\in R$ and $D_0$ is $R$-linear.
Since $D_i (tD_0-\alpha)m\in \Ker (tD_0-\alpha)^j$ and $m\in \Ker (tD_0-\alpha)^j$, the last sum is zero.
\end{proof}

We continue the proof of Proposition~\ref{prop:constancy of residue}.
Fix an $R[[t]]$-basis of $\Lambda_{\calA}$ and identify $\Lambda_{\calA}$
with $R[[t]]^r$. Note that $R[[t]]^r$ has natural differentials $d_0,d_1,\ldots,d_n\colon R[[t]]^r\ra R[[t]]^r$.
Consider the map
\[
 t (D_0-d_0)\colon R[[t]]^r\ra R[[t]]^r.
\]
This is $R[[t]]$-linear.
We denote the corresponding $r\times r$ matrix by $C_0 \in M_r\bigl(R[[t]]\bigr)$.

Fix $1\leq i\leq n$.
By Lemma~\ref{lem:ker lattice is D-stable}, the map $D_i$ gives an endomorphism on $R[[t]]^r$ that satisfies the Leibniz rule
and thus $D_i-d_i$ is an $R[[t]]$-linear endomorphism on $R[[t]]^r$.
We denote the corresponding $r\times r$ matrix by $C_i \in M_r\bigl(R[[t]]\bigr)$.

We have $[tD_0,D_i]=0$ and $[td_0,d_i]=0$ in $\End \bigl(R[[t]]^r\bigr)$.
Plugging $tD_0=td_0+C_0$ and $D_i=d_i+C_i$ into $[tD_0,D_i]=0$ yields
\begin{equation}\label{eq:commutator relation for residue}
 [C_0,C_i]=d_iC_0-td_0C_i,
\end{equation}
where $d_0$ and $d_i$ are differentials acting on the matrices entrywise.

Consider the surjection $R[[t]]\ra R$ evaluating $t$ by 0.
We denote the image of $C_0$ (resp.~$C_i$) in $M_r(R)$ by $\overline{C}_0$
(resp.~$\overline{C}_i$).
By construction $\overline{C}_0$ is the matrix corresponding to $\Res_{\Lambda_{\calA}}D_0$. 
Thus it suffices to show that each eigenvalue of $\overline{C}_0$ is killed by $d_i$. This is standard.
Namely, from (\ref{eq:commutator relation for residue}), 
we have
\[
  \Bigl[\overline{C}_0,\overline{C}_i\Bigr]=d_i\overline{C}_0.
\]

This implies that 
\[
 d_i (\overline{C}_0^2)=\overline{C}_0d_i(\overline{C}_0)+d_i(\overline{C}_0)\overline{C}_0=\overline{C}_0\Bigl[\overline{C}_0,\overline{C}_i\Bigr]+\Bigl[\overline{C}_0,\overline{C}_i\Bigr]\overline{C}_0
=\Bigl[\overline{C}_0^2,\overline{C}_i\Bigr].
\]
Similarly, for each $j\in \N$,
\[
 d_i (\overline{C}_0^j)=\Bigl[\overline{C}_0^j,\overline{C}_i\Bigr].
\]
In particular, we get
\[
 d_i\bigl(\tr(\overline{C}_0^j)\bigr)=0.
\]
This implies that each eigenvalue of $\overline{C}_0$ is killed by $d_i$.
\end{proof}

\subsection{Constancy of generalized Hodge-Tate weights}

Here is the key theorem of this paper.

\begin{thm}\label{thm:constancy of eigenvalues of arithmetic Sen operator}
Let $k$ be a finite extension of $\Q_p$.
Let $X$ be a smooth rigid  analytic variety over $k$
and $\bL$ a $\Q_p$-local system on $X_{\et}$.
Consider the arithmetic Sen endomorphism $\phi_{\bL}\in \End\bigl(\calH(\bL)\bigr)$.
Then eigenvalues of $\phi_{\bL,x}\in \End\bigl(\calH(\bL)_x\bigr)$ for $x\in X_K$ are algebraic over $k$ and constant on each connected component of $X_K$.
\end{thm}

We call these eigenvalues \textit{generalized Hodge-Tate weights} of $\bL$.

\begin{proof}
Since $\phi_{\bL}$ is an endomorphism on the vector bundle $\calH(\bL)$ on $X_K$, it suffices to prove the statement \'etale locally on $X$.
Thus we may assume that $X$ is an affinoid $\Spa (B,B^+)$ which admits
a toric chart $X_{k'}\ra \bT^n_{k'}$ over some finite extension $k'$ of $k$ in $K$.

Take $(Y=\Spa(B,B^+)\ra X_{k'})\in\calB_{\bL}$.
We may assume that $B_{\infty}$ is connected, hence an integral domain.
Note that $Y$ admits a toric chart 
\[
 Y_{k''}\ra \bT^n_{k''}
=\Spa (k''\langle T_1^\pm,\ldots,T_n^\pm\rangle,\calO_{k''}\langle T_1^\pm,\ldots,T_n^\pm\rangle)
\]
after base change to a finite extension $k''$ of $k'$ in $K$.
Then the derivations $\frac{\partial}{\partial T_1},\ldots,\frac{\partial}{\partial T_n}$ on $k''\langle T_1^\pm,\ldots,T_n^\pm\rangle$
extends over $B_{\infty}$. We also denote the extensions by 
$\frac{\partial}{\partial T_1},\ldots,\frac{\partial}{\partial T_n}$.

We set
\[
 R=B_{k_\infty}, \qquad d_0=\partial_t, \quad \text{and}\quad
d_i=\frac{\partial}{\partial T_i} \quad (1\leq i\leq n).
\]

Consider $R((t))$-module 
\[
 M=\RH(\bL)(Y_K)_{\fin}
\]
equipped with endomorphisms
\[
 D_0=t^{-1}\phi_{\dR,\bL,Y_K}, \quad \text{and} \quad
D_i=(\nabla_{\bL, Y_K})_{\frac{\partial}{\partial T_i}} \quad (1\leq i\leq n).
\]
By Proposition~\ref{prop:commutativity of de Rham Sen and connection}, they satisfy the assumptions in the previous subsection 

Consider the $R[[t]]$-submodule of $M$
\[
 \Lambda=(\Fil^0\RH(\bL)(Y_K))_{\fin}.
\]
Then $\Lambda$ is $tD_0$-stable, and 
$\Res_{\Lambda}D_0$ is $\phi_{\bL,Y_K}$.
Thus by Proposition~\ref{prop:constancy of residue},
each eigenvalue $\alpha$ of $\Res_{\Lambda}D_0$ 
in an algebraic closure $L$ of $\Frac R$
satisfies
\[
 d_1(\alpha)=\cdots=d_n(\alpha)=0.
\]

On the other hand, we can check that
\[
L^{d_1=\cdots=d_n=0}
=\Bigl(\overline{\Frac k''\langle T_1^\pm,\ldots,T_n^\pm\rangle}\Bigr)^{\frac{\partial}{\partial T_1}=\cdots=\frac{\partial}{\partial T_n}=0}= \bar{k}.
\]
Therefore the eigenvalues of $\phi_{\bL,Y_K}$ are 
algebraic over $k$ and constant on $Y_K$.
 \end{proof}

 \begin{cor}\label{cor:constancy of generalized HT weights}
Let $k$ be a finite extension of $\Q_p$.
Let $X$ be a geometrically connected smooth rigid  analytic variety over $k$
and $\bL$ a $\Q_p$-local system on $X$.
Then the multiset of generalized Hodge-Tate weights
of the $p$-adic representations $\bL_{\overline{x}}$ of $\Gal\bigl(\overline{k(x)}/k(x)\bigr)$ does not depend on the choice of a classical point $x$ of $X$.

In particular, if $\bL_{\overline{x}}$ is presque Hodge-Tate
for one classical point $x$ of $X$ (i.e., generalized Hodge-Tate weights are all integers), $\bL_{\overline{y}}$ is presque Hodge-Tate for every classical point $y$ of $X$.
\end{cor}

\begin{proof}
This follows from Theorem~\ref{thm:constancy of eigenvalues of arithmetic Sen operator}.
\end{proof}

\section{Applications and related topics}
\label{section:applications}
We study properties of Hodge-Tate sheaves
using the arithmetic Sen endomorphism.
We keep the notation in Section~\ref{section:arithmetic Sen endomorphism}.

Consider the Hodge-Tate period sheaf on $X_{\proet}$
\[
 \OB_{\HT}:=\gr^\bullet \OB_{\dR}=\bigoplus_{j\in\Z} \OC(j).
\]
For a $\Q_p$-local system $\bL$ on $X_{\et}$, we define
a sheaf $D_{\HT}(\bL)$ on $X_{\et}$ by
\[
 D_{\HT}(\bL):= \nu_\ast (\hat{\bL}\otimes_{\hat{\Q}_p}\OB_{\HT}).
\]

\begin{prop}
The sheaf $D_{\HT}(\bL)$ is a coherent $\calO_{X_{\et}}$-module.
Moreover, for every affinoid $Y\in X_{\et}$,
\[
 \Gamma(Y,D_{\HT}(\bL))
=\bigoplus_{j\in\Z}H^0(\Gamma_k,\calH(\bL)(j)).
\]
\end{prop}

\begin{proof}
 This follows from the proof of \cite[Theorem 3.9 (i)]{LZ}.
\end{proof}

\begin{rem}
In \cite[Theorem 8.6.2 (a)]{KL-II}.
Kedlaya and Liu proved this statement for pseudocoherent modules
over a pro-coherent analytic field.
\end{rem}

We are going to study the relation between $D_{\HT}(\bL)$ and 
$\phi_{\bL}\in \End \calH(\bL)$.
For each $j\in\Z$, we set
\[
 \calH(\bL)^{\phi_{\bL}=j}:=\Ker 
\bigl(\phi_{\bL}-j\id\colon \calH(\bL)\ra \calH(\bL)\bigr).
\]
This is a coherent $\calO_{X_{K,\et}}$-module.
We denote by $D_{\HT}(\bL)|_{X_K}$ the coherent $\calO_{X_K,\et}$-module
associated to the pullback of $D_{\HT}(\bL)$ on $X$ to $X_K$ as coherent sheaves.

\begin{prop}\label{prop:DHT and kernel of Sen}
Let $\bL$ be a $\Q_p$-local system of rank $r$ on $X_{\et}$.
 Assume that $\bL$ satisfies one of the following conditions:
\begin{enumerate}
 \item $\calH(\bL)^{\phi_{\bL}=j}$ is a vector bundle
on $X_{K,\et}$ for each $j\in \Z$.
 \item $D_{\HT}(\bL)$ is a vector bundle of rank $r$ on $X_{\et}$.
\end{enumerate}
Then we have
\[
 D_{\HT}(\bL)|_{X_{K,\et}}\cong
\bigoplus_{j\in\Z}\calH(\bL(j))^{\phi_{\bL(j)}=0}.
\]
Moreover, this is isomorphic to 
$\bigoplus_{j\in\Z}\calH(\bL)^{\phi_{\bL}=j}$.
In particular, 
$D_{\HT}(\bL)$ is a vector bundle on $X_{\et}$
and $\bigoplus_{j\in\Z}\calH(\bL)^{\phi_{\bL}=j}$
is a vector bundle on $X_{K,\et}$.
\end{prop}

\begin{proof}
The statement is local. 
So it suffices to prove that
for each affinoid $Y=\Spa(B,B^+)\in X_{\et}$
such that $\calH(\bL)|_{Y_K}$ is 
associated to a finite free $B_K$-module (say $V$),
we have
\[
 \Gamma(Y,D_{\HT}(\bL))\hat{\otimes}_B B_K
 \cong \bigoplus_{j\in\Z}V(j)^{\phi_{V(j)}=0}.
\]

 Note $\Gamma(Y,D_{\HT}(\bL))=\bigoplus_{j\in\Z} \bigl(V(j)\bigr)^{\Gamma_k}$.
Moreover, it follows from the Tate-Sen method (\cite[Lemma 3.10]{LZ})
that
\[
 \bigl(V_{\fin}(j)\bigr)^{\Gamma_k}
\stackrel{\cong}{\lra} \bigl(V(j)\bigr)^{\Gamma_k}.
\]

\begin{lem}\label{lem:Gamma-inv is inj}
\hfill
\begin{enumerate}
 \item 
$\bigl(V_{\fin}^{\phi_V=0}\bigr)\otimes_{B_{\infty}}B_K\cong V^{\phi_V=0}$.
 \item
The natural map
\[
 (V_{\fin}^{\Gamma_k})\otimes_B B_{\infty} \ra V_{\fin}^{\phi_V=0}
\]
is injective.
\end{enumerate}
\end{lem}

\begin{proof}
 Part (i) follows from the flatness of $B_{\infty}\ra B_K$
and $V_{\fin}\otimes_{B_{\infty}}B_K\cong V$.

We prove part (ii).
By the definition of $\phi_V$, the natural map
\[
  (V_{\fin}^{\Gamma_k})\otimes_B B_{\infty} \ra V_{\fin}
\]
factors through $V_{\fin}^{\phi_V=0}$. 
So we show that the above map is injective.

We denote the total fraction ring of $B$ (resp.~$B_{\infty}$)
by $\Frac B$ (resp.~$\Frac B_{\infty}$).
We first claim that the natural map
\[
 V_{\fin}^{\Gamma_k} \ra V_{\fin}^{\Gamma_k}\otimes_B \Frac B
\]
is injective.
To see this, note that $V_{\fin}$ is a finite free $B_{\infty}$-module.
Hence the composite
\[
 V_{\fin}\ra V_{\fin}\otimes_B \Frac B
= V_{\fin}\otimes_{B_{\infty}}(B_{\infty}\otimes_B \Frac B)
\ra V_{\fin}\otimes_{B_{\infty}}\Frac B_{\infty}
\]
is injective, and thus so is the first map.
Since the composite
\[
 V_{\fin}^{\Gamma_k}\ra V_{\fin}^{\Gamma_k}\otimes_B \Frac B
\ra V_{\fin}\otimes_B \Frac B
\]
coincides with the composite of injective maps
$V_{\fin}^{\Gamma_k}\ra V_{\fin}$ and $V_{\fin}\ra V_{\fin}\otimes_B \Frac B$,
the map $V_{\fin}^{\Gamma_k} \ra V_{\fin}^{\Gamma_k}\otimes_B \Frac B$ is also injective.

By the above claim, it suffices to show the injectivity of the natural map
\[
 (V_{\fin}^{\Gamma_k}\otimes_B \Frac B) \otimes_{\Frac B} \Frac B_{\infty}
\ra V_{\fin}\otimes_{B_{\infty}} \Frac B_{\infty}.
\]
Now that $\Frac B$ and $\Frac B_{\infty}$ are products of fields,
this follows from standard arguments;
we may assume that $\Frac B$ is a field. Replacing $k$ by an algebraic closure in $\Frac B$, we may further assume that $\Frac B_{\infty}$ is also a field.
Note that $\Frac B_{\infty}=(\Frac B)\otimes_k k_\infty$ and thus
$(\Frac B_{\infty})^{\Gamma_k}=\Frac B$.

Assume the contrary. Take $v_1,\ldots,v_a \in V_{\fin}^{\Gamma_k}\otimes_B \Frac B$ and  $b_1,\ldots, b_a\in \Frac B_{\infty}$
such that $v_1,\ldots,v_a$ are linearly independent over $\Frac B$
and $b_1 v_1+\cdots+b_a v_a=0$.
We may assume that they are of minimum length.
Then $b_1\neq 0$ and thus replacing $b_i$ by $b_1^{-1}b_i$ 
we may further assume $b_1=1$. It is obvious when $a=1$, so we assume $a>1$.
Take any $\gamma\in \Gamma_k$. As $v_1,\ldots,v_a \in V_{\fin}^{\Gamma_k}\otimes_B \Frac B$, we have
$v_1+\gamma(b_2)v_2+\cdots \gamma(b_a)v_a=0$ and thus
$(\gamma(b_2)-b_2)v_2+\cdots +(\gamma(b_a)-b_a)v_a=0$.
By the minimality, we have $\gamma(b_i)=b_i$
for each $2\leq i\leq a$ and $\gamma\in \Gamma_k$.
Therefore we have $b_i\in \Frac B$ for all $i$, which contradicts
the linear independence of $v_1,\ldots,v_a$ over $\Frac B$.
\end{proof}

We continue the proof of Proposition~\ref{prop:DHT and kernel of Sen}.
By Lemma~\ref{lem:Gamma-inv is inj} and discussions above,
it is enough to show 
$V_{\fin}^{\Gamma_k}\otimes_B B_{\infty}\cong V_{\fin}^{\phi_V=0}$
assuming either condition (i) or (ii).
In fact, the Tate twist of this isomorphism implies 
$(V(j))^{\Gamma_k}\hat{\otimes}_B B_K\cong (V(j))^{\phi_{V(j)}=0}$,
and a choice of a generator of $\calO_{X_{K,\et}}(j)$
yields $\calH(\bL(j))^{\phi_{\bL(j)}=0}\cong \calH(\bL)^{\phi_{\bL}=-j}$.

We show that condition (ii) implies condition (i).
For each $j\in\Z$, let $V_{\fin}^{(j)}$
denote the generalized eigenspace of $\phi_V$ on $V_{\fin}$
with eigenvalue $j$.
By the constancy of $\phi_V$, 
$V_{\fin}^{(j)}$ is a direct summand of $V_{\fin}$
and thus a finite projective $B_{\infty}$-module.
By Lemma~\ref{lem:Gamma-inv is inj}(ii), we have injective $B_{\infty}$-linear maps
\[
 \bigl(V(j)_{\fin}\bigr)^{\Gamma_k}\otimes_B B_{\infty}
\hra \bigl(V(j)_{\fin}\bigr)^{\phi_{V(j)}=0}
\cong V_{\fin}^{\phi_V=-j}
\hra V_{\fin}^{(-j)}
\]
for each $j\in \Z$. From this we obtain
\[
 \rank D_{HT}(\bL)=\sum_{j\in\Z}\rank_B \bigl(V_{\fin}(j)\bigr)^{\Gamma_k}
\leq \sum_{j\in\Z} \rank_{B_{\infty}}V_{\fin}^{(-j)}
\leq \rank \calH(\bL)=r.
\]
Hence it follows from condition (ii) that
$\bigl(V_{\fin}(j)\bigr)^{\Gamma_k}\otimes_B B_{\infty}$
and $V_{\fin}^{(-j)}$ are finite projective $B_{\infty}$-modules of the same rank.
This implies $V_{\fin}^{\phi_V=-j}=V_{\fin}^{(-j)}$. 
So $V_{\fin}^{\phi_V=-j}$ is a finite projective $B_{\infty}$-module for every $j\in \Z$ and thus $\calH(\bL)$ satisfies condition (i).

From now on, we assume that $\calH(\bL)$ satisfies condition (i).
By condition (i) and Lemma~\ref{lem:Gamma-inv is inj}(i), 
$V_{\fin}^{\phi_V=0}$ is finite projective over $B_{\infty}$.
So shrinking $Y$ if necessary, we may assume that $V_{\fin}^{\phi_V=0}$ is finite free over $B_{\infty}$.
Note that we only concern the $B_{\infty}$-representation $V_{\fin}$ of $\Gamma_k$ and we have $(V_{\fin}^{\phi_V=0})^{\Gamma_k}=V_{\fin}^{\Gamma_k}$.
Thus replacing $V_{\fin}$ by the subrepresentation $V_{\fin}^{\phi_V=0}$,
we may further assume $\phi_V=0$ on $V_{\fin}$.
Under this assumption, it remains to prove
$V_{\fin}^{\Gamma_k}\otimes_B B_{\infty}\cong V_{\fin}$.

Fix a $B_{\infty}$-basis $v_1,\ldots, v_r$ of $V_{\fin}$.
Then there exists a large positive integer $m$
such that for each $\gamma\in \Gamma_k$ the matrix of $\gamma$ 
with respect to $(v_i)$ has entries in $\GL_r(B_{k_m})$.
Since $\phi_V=0$, by increasing $m$ if necessary,
we may further assume that 
$\gamma v_i=v_i$ for each $1\leq i\leq r$
and $\gamma\in \Gamma_k':=\Gal(k_\infty/k_m)\subset \Gamma_k$.
Set $V_{k_m}:=\bigoplus_{1\leq i\leq r}B_{k_m}v_i$.
This is a $B_{k_m}$-representation of $\Gamma_k/\Gamma_k'=\Gal(k_m/k)$
and satisfies $V_{\fin}=V_{k_m}\otimes_{B_{k_m}}B_{\infty}$.

It follows from \cite[Proposition 2.2.1]{Berger-Colmez}
that $(V_{k_m})^{\Gamma_k/\Gamma_k'}$ is a finite projective $B$-module
and that $(V_{k_m})^{\Gamma_k/\Gamma_k'}\otimes_B B_{k_m}\cong V_{k_m}$.
As $V_{\fin}^{\Gamma_k}=(V_{k_m})^{\Gamma_k/\Gamma_k'}$,
this yields
\[
 V_{\fin}^{\Gamma_k}\otimes_B B_{\infty}\cong V_{\fin}.
\]
\end{proof}

\begin{thm}\label{thm:def of Hodge-Tate sheaves}
Let $\bL$ be a $\Q_p$-local system of rank $r$ on $X_{\et}$.
Then the following conditions are equivalent:
\begin{enumerate}
 \item $D_{\HT}(\bL)$ is a vector bundle of rank $r$ on $X_{\et}$.
 \item 
$\nu^\ast D_{\HT}(\bL)\otimes_{\calO_X}\OB_{\HT}
\cong \hat{\bL}\otimes_{\hat{\Q}_p}\OB_{\HT}$.
 \item $\phi_{\bL}$ is a semisimple endomorphism on $\calH(\bL)$
with integer eigenvalues.
 \item There exist integers $j_1<\ldots <j_a$ such that
if we set $F(s):=\prod_{1\leq i\leq a}(s-j_i)\in \Z[s]$, then
\[
 F(\phi_{\bL})=0
\]
as endomorphism of $\calH(\bL)$.
\end{enumerate}
\end{thm}

\begin{defn}
A $\Q_p$-local system on $X_{\et}$ is a \textit{Hodge-Tate sheaf}
if it satisfies the equivalent conditions in Theorem~\ref{thm:def of Hodge-Tate sheaves}.
\end{defn}

\begin{rem}
 Tsuji obtained Theorem~\ref{thm:def of Hodge-Tate sheaves}
in the case of semistable schemes (\cite[Theorem~9.1]{Tsuji}).
He also gave a characterization of Hodge-Tate local systems
in terms of restrictions to divisors. See \textit{loc.~cit.} for the detail.
\end{rem}

\begin{proof}[Proof of Theorem~\ref{thm:def of Hodge-Tate sheaves}]
The equivalence of (iii) and (iv) is clear,
and (iii) implies (i) by Proposition~\ref{prop:DHT and kernel of Sen}.
Conversely, assume condition (i).
Thus $D_{\HT}(\bL)|_{X_K}$ is a vector bundle of rank $r$ on $X_{K,\et}$.
By Proposition~\ref{prop:DHT and kernel of Sen},
it is also isomorphic to $\bigoplus_{j\in\Z}\calH(\bL)^{\phi_{\bL}=j}$.
Thus $\bigoplus_{j\in\Z}\calH(\bL)^{\phi_{\bL}=j}=\calH(\bL)$,
and there exist integers $j_1<\ldots< j_a$
such that $\bigoplus_{1\leq i\leq a}\calH(\bL)^{\phi_{\bL}=j_i}=\calH(\bL)$.
So $F(s):=\prod_{1\leq i\leq a}(s-j_i)$
satisfies $F(\phi_{\bL})=0$, which is condition (iv).

Next we show that condition (iv) implies (ii).
Obviously, there is a natural morphism
\begin{equation}\label{eq:comp map on proetale site}
 \nu^\ast D_{\HT}(\bL)\otimes_{\calO_X}\OB_{\HT}
\ra \hat{\bL}\otimes_{\hat{\Q}_p}\OB_{\HT}
\end{equation}
on $X_{\proet}$ and we will prove that this is an isomorphism.
It is enough to check this on $X_{\proet}/X_K\cong X_{K,\proet}$.
Recall a canonical isomorphism in \cite[Theorem 2.1(ii)]{LZ}:
\[
\nu^\ast \calH(\bL)\otimes_{\calO_{X_K}}\OC|_{X_{K,\proet}}
\cong (\hat{\bL}\otimes_{\hat{\Q}_p}\OC)|_{X_{K,\proet}}. 
\]
Then the restriction of the morphism (\ref{eq:comp map on proetale site})
to $X_{K,\proet}$ is obtained as
\begin{align*}
 (\nu^\ast D_{\HT}(\bL)\otimes_{\calO_X}\OB_{\HT})|_{X_{K,\proet}}
&\cong 
\nu^\ast D_{\HT}(\bL)|_{X_{K,\proet}}
\otimes_{\calO_{X_K}}\OB_{\HT}|_{X_{K,\proet}}\\
&\cong \nu'^\ast(D_{\HT}(\bL)|_{X_K})\otimes_{\calO_{X_K}}
\bigl(\bigoplus_{j\in\Z}\OC(j)\bigr)|_{X_{K,\proet}}\\
&\cong \nu'^\ast\bigl(
D_{\HT}(\bL)|_{X_K}\otimes \bigoplus_{j\in\Z}\calO_{X_K}(j)\bigr)
\otimes_{\calO_{X_K}}\OC|_{X_{K,\proet}}\\
&\ra \nu'^\ast\bigl(
\bigoplus_{j\in\Z}\calH(\bL)(j)\bigr)\otimes_{\calO_{X_K}}
\OC|_{X_{K,\proet}}\\
&\cong \bigl(\bigoplus_{j\in\Z}
\hat{\bL}\otimes_{\hat{\Q}_p}\OC(j)\bigr)|_{X_{K,\proet}}
\cong (\hat{\bL}\otimes_{\hat{\Q}_p}\OB_{\HT})|_{X_{K,\proet}}.
\end{align*}
This can be checked by considering affinoid perfectoids
represented by the toric tower, and 
the verification is left to the reader.
It follows from condition (iii) and Proposition~\ref{prop:DHT and kernel of Sen}
that 
\[
 \bigoplus_{j\in\Z}D_{\HT}(\bL)|_{X_K}(j)
\cong \bigoplus_{j\in\Z}\calH(\bL)(j).
\]
Hence 
$(\nu^\ast D_{\HT}(\bL)\otimes_{\calO_X}\OB_{\HT})|_{X_{K,\proet}}
\cong (\hat{\bL}\otimes_{\hat{\Q}_p}\OB_{\HT})|_{X_{K,\proet}}$.

Finally we show that (ii) implies (i).
By condition (ii), we have
\[
 \nu'_\ast\bigl((\nu^\ast D_{\HT}(\bL)\otimes_{\calO_X}\OB_{\HT})|_{X_{K,\proet}}\bigr)
\cong \nu'_\ast\bigl((\hat{\bL}\otimes_{\hat{\Q}_p}\OB_{\HT})|_{X_{K,\proet}}\bigr).
\]
On the other hand, it is easy to check 
\[
 \nu'_\ast\bigl((\nu^\ast D_{\HT}(\bL)\otimes_{\calO_X}\OB_{\HT})|_{X_{K,\proet}}\bigr)
\cong \bigoplus_{j\in\Z}D_{\HT}(\bL)|_{X_K}(j).
\]
Since $\nu'_\ast\bigl((\hat{\bL}\otimes_{\hat{\Q}_p}\OB_{\HT})|_{X_{K,\proet}}\bigr)=\bigoplus_{j\in\Z}\calH(\bL(j))$, we have
\[
 \bigoplus_{j\in\Z}D_{\HT}(\bL)|_{X_K}(j)\cong 
\bigoplus_{j\in\Z}\calH(\bL(j)).
\]
In particular, 
$D_{\HT}(\bL)|_{X_K}$ is a vector bundle on $X_{K,\et}$,
and thus $D_{\HT}(\bL)$
is a vector bundle on $X_{\et}$.
Moreover, condition (ii) implies $\rank D_{\HT}(\bL)=r$.
\end{proof}

\begin{example}
Suppose that there exists a Zariski dense subset $T\subset X$ consisting
of classical rigid points with residue field finite over $k$
such that the restriction of $\bL$ to each $x\in T$
defines a Hodge-Tate representation.
Then $\bL$ is a Hodge-Tate sheaf 
by Theorem~\ref{thm:constancy of eigenvalues of arithmetic Sen operator} 
and Theorem~\ref{thm:def of Hodge-Tate sheaves}(iii).
See \cite[Theorem 8.6.6]{KL-II} for a generalization of this remark.
\end{example}

\begin{cor}
\hfill
\begin{enumerate}
 \item Hodge-Tate sheaves are stable under taking dual, tensor product,
and subquotients.
 \item Let $f\colon Y\ra X$ be a morphism between
smooth rigid analytic varieties over $k$.
If $\bL$ is a Hodge-Tate sheaf on $X_{\et}$, then 
$f^\ast{\bL}$ is a Hodge-Tate sheaf on $Y_{\et}$.
\end{enumerate}
\end{cor}

\begin{proof}
This follows from Proposition~\ref{prop:Sen endom}, 
Lemma~\ref{lem:properties of Sen endom}, 
and Theorem~\ref{thm:def of Hodge-Tate sheaves}(iii).
\end{proof}

We next turn to the pushforward of Hodge-Tate sheaves.

\begin{thm}\label{thm:pushforward of Hodge-Tate sheaves}
Let $f \colon X\ra Y$ be a smooth proper morphism 
between smooth rigid analytic varieties over $k$
of relative dimension $m$
and let $\bL$ be a $\Z_p$-local system on $X_{\et}$.
Assume that $R^if_\ast \bL$ is a $\Z_p$-local system on $Y_{\et}$
for each $i$.
\begin{enumerate}
 \item If $\alpha\in \overline{k}$ is a generalized Hodge-Tate
for $R^if_\ast \bL$,
then $\alpha$ is of the form $\beta-j$
with a generalized Hodge-Tate weight $\beta$ of $\bL$
and an integer $j\in [0,m]$.
 \item If $\bL$ is a Hodge-Tate sheaf on $X_{\et}$
\footnote{This means that the $\Q_p$-local system $\bL\otimes_{\Z_p}\Q_p$ is a Hodge-Tate sheaf on $X_{\et}$.}, then 
$R^if_\ast \bL$ is a Hodge-Tate sheaf on $Y_{\et}$.
\end{enumerate}
\end{thm}

\begin{rem}
Theorem~\ref{thm:pushforward of Hodge-Tate sheaves}(ii)
is proved by Hyodo (\cite[\S 3, Corollary]{Hyodo})
when $f\colon X\ra Y$ and $\bL$ are analytifications of corresponding algebraic objects.
\end{rem}

\begin{proof}
Let $f_K\colon X_K\ra Y_K$ denote the base change of $f$ over $K$.

Part (i) easily follows from 
Theorem~\ref {thm:pushforward of Sen endomorphism}.
In fact, we have the isomorphism
\[
 \calH(R^if_\ast \bL) \cong R^if_{K,\et,\ast}
\bigl(
\calH(\bL)\otimes \Omega_{X/Y}^\bullet(-\bullet) \bigr),
\]
and under this identification
$\phi_{R^if_\ast\bL}$ corresponds to 
$R^if_{K,\et,\ast}(\phi_{\bL}\otimes\id -\bullet (\id\otimes\id))$.
Consider the spectral sequence with 
\[
 E_1^{a,b}=R^bf_{K,\et,\ast}\calH(\bL)\otimes \Omega_{X/Y}^a(-a)
\]
converging to $\calH(R^{a+b}f_\ast \bL)$.
Then the endomorphism $R^bf_{K,\et,\ast}((\phi_{\bL}-a)\otimes \id)$
on $E_1^{a,b}$ converges to $\phi_{R^{a+b}f_\ast\bL}$, and
this implies part (i).

For part (ii), we need arguments similar to the proof of Theorem~\ref {thm:pushforward of Sen endomorphism}. 
We may assume that $Y$ is affinoid.
Take a finite affinoid covering
$\calU=\{U_K^{(i)}\}$ of $X_K$.
Let $\calF^\bullet$ denote
the complex of $\calO_{X_K}$-modules 
\[
 \calH(\bL)\stackrel{\vartheta_{\bL}}{\lra}
\calH(\bL)\otimes\Omega_{X/Y}^1(-1)
\stackrel{\vartheta_{\bL}}{\lra}
\calH(\bL)\otimes\Omega_{X/Y}^2(-2)\lra \cdots
\]
on $X_K$ equipped with the natural $\Gamma_k$-action and the endomorphism
$\phi_{\calF^\bullet}=\phi_\bL\otimes\id - \bullet(\id\otimes \id)$.

Recall also 
the \v{C}ech-to-derived functor spectral sequence 
 with
\[
 E_2^{a,b}=H^a(\Tot(\check{C}^\bullet(\calU,\underline{H}^b(\calF^\bullet))))
\]
converging to $R^{a+b}\Gamma(X_{K,\et},\calF^\bullet)$.
This spectral sequence degenerates at $E_2$ and yields 
\begin{equation}\label{eq:Cech description}
  H^i(\Tot(\check{C}^\bullet(\calU,\calF^\bullet)))
\stackrel{\cong}{\lra}
R^i\Gamma(X_{K,\et},\calF^\bullet)
=\Gamma(Y_K, \calH(R^if_\ast \bL)).
\end{equation}
Note that both source and target in (\ref{eq:Cech description})
have arithmetic Sen endomorphisms and they are compatible under the isomorphism.

Since $\bL$ is a Hodge-Tate sheaf, 
there exist integers $j_1<\ldots <j_a$ such that
$F(\phi_{\bL})=0$
with $F(s):=\prod_{1\leq i\leq a}(s-j_i)$.
Set $J:=\{j_1-m, j_1-m+1,\ldots, j_a-1, j_a\}$.
This is a finite subset of $\Z$.
We set $G(s):=\prod_{j\in J}(s-j)\in \Z[s]$.
For each $0\leq j\leq m$, 
the endomorphism $\phi_{\bL}\otimes \id - j(\id\otimes \id)$
on $\calH(\bL)\otimes \Omega^j_{X/Y}(-j)$ 
satisfies
\[
 G(\phi_{\bL}\otimes \id - j(\id\otimes \id))=0.
\]
This implies $G(\phi_{\calF^\bullet})=0$, and thus
$G(\Tot(\check{C}^\bullet(\calU,\phi_{\calF^\bullet})))=0$.
Therefore (\ref{eq:Cech description}) yields
\[
 G(\phi_{R^if_\ast\bL})=0.
\]
Hence $R^if_\ast\bL$ is a Hodge-Tate sheaf on $Y_{\et}$.
\end{proof}

We now turn to a rigidity of Hodge-Tate representations.
Let us first recall Liu and Zhu's rigidity result for de Rham representations (\cite[Theorem 1.3]{LZ}):
let $X$ be a geometrically connected smooth rigid  analytic variety over $k$
and let $\bL$ be a $\Q_p$-local system on $X_{\et}$.
If $\bL_{\overline{x}}$ is a de Rham representation at a classical point $x\in X$, then $\bL$ is a de Rham sheaf. In particular, $\bL_{\overline{y}}$ is a de Rham representation at every classical point $y\in X$.

The same result holds for Hodge-Tate local systems of rank at most $2$.
We do not know whether this is true for Hodge-Tate local systems of higher rank.

\begin{thm}\label{thm:principle B for rank 2}
Let $k$ be a finite extension of $\Q_p$.
Let $X$ be a geometrically connected smooth rigid analytic variety over $k$
and let $\bL$ be a $\Q_p$-local system on $X_{\et}$.
Assume that $\rank \bL$ is at most two.
If $\bL_{\overline{x}}$ is a Hodge-Tate representation at a classical point $x\in X$, then $\bL$ is a Hodge-Tate sheaf.
In particular, $\bL_{\overline{y}}$ is a Hodge-Tate representation at every classical point $y\in X$
\end{thm}

Before the proof, let us recall a remarkable theorem by Sen
on Hodge-Tate representations of weights $0$.

\begin{thm}[{\cite[\S Corollary]{Sen-cont}}]\label{thm:Sen}
 Let $k$ be a finite extension of $\Q_p$ and let
$\rho\colon G_k\ra \GL_r(\Q_p)$ be a continuous representation of
the absolute Galois group of $k$.
Then $\rho$ is a Hodge-Tate representations with all the Hodge-Tate weights zero
if and only if $\rho$ is potentially unramified, i.e., the image of 
the inertia subgroup of $k$ is finite.
\end{thm}

Note that $\rho$ being a Hodge-Tate representations with all the Hodge-Tate weights zero is equivalent to the Sen endomorphism of $\rho$ being zero.
Since potentially unramified representations are de Rham and 
de Rham representations are stable under Tate twists,
Theorem~\ref{thm:Sen} implies that a Hodge-Tate representation with a single weight is necessarily de Rham.

\begin{proof}[Proof of Theorem~\ref{thm:principle B for rank 2}]
 We check condition (iii) in Theorem~\ref{thm:def of Hodge-Tate sheaves}.
By Theorem~\ref{thm:constancy of eigenvalues of arithmetic Sen operator} and assumption, 
all the eigenvalues of $\phi_{\bL}$ are integers.
So the statement is obvious either when $\rank \bL=1$ or when $\rank \bL=2$ and
two eigenvalues are distinct integers.

Assume that $\rank \bL=2$ and two eigenvalues are the same integer.
Then $\bL_{\overline{x}}$ is de Rham by Theorem~\ref{thm:Sen}, 
and thus $\bL$ is de Rham by ``Principle B'' for de Rham sheaves (\cite[Theorem 1.3]{LZ}). In particular, $\bL$ is a Hodge-Tate sheaf.
\end{proof}

\begin{rem}
The proof shows that Theorem~\ref{thm:principle B for rank 2} holds for $\bL$ 
of an arbitrary rank if one of the following conditions holds:
\begin{enumerate}
 \item $\bL_{\overline{x}}$ is a Hodge-Tate representation with a single weight at a classical point $x\in X$.
 \item $\bL_{\overline{x}}$ is a Hodge-Tate representation with $\rank \bL$ distinct weights at a classical point $x\in X$.
\end{enumerate}
\end{rem}

We end with another application of Sen's theorem in the relative setting.

\begin{thm}\label{thm:weight zero p-adic monodromy}
Let $k$ be a finite extension of $\Q_p$.
Let $X$ be a smooth rigid analytic variety over $k$
and let $\bL$ be a $\Z_p$-local system on $X_{\et}$.
Assume that $\bL$ is a Hodge-Tate sheaf with a single Hodge-Tate weight.
Then there exists a finite \'etale cover $f\colon Y\ra X$ such that
$(f^\ast \bL)_{\overline{y}}$ is semistable at every classical point 
$y$ of $Y$.
 \end{thm}

\begin{proof}
Since semistable representations are stable under Tate twists,
we may assume that $\bL$ is a Hodge-Tate sheaf with all the weights zero.
Let $\overline{\bL}$ denote the $\Z/p^2$-local system $\bL/p^2\bL$ on $X_{\et}$.
Then there exists a finite \'etale cover $f\colon Y\ra X$ such that
$f^\ast\overline{\bL}$ is trivial on $Y_{\et}$.
We will prove that this $Y$ works.

Let $y$ be a classical point of $Y$.
We denote by $k'$ the residue field of $y$.
Let $\rho\colon G_{k'}\ra \GL(V)$ be the Galois representation
of $k'$ corresponding to the stalk $V:=(f^\ast\overline{\bL})_{\overline{y}}$
at a geometric point $\overline{y}$ above $y$.
By assumption, $\rho$ is a Hodge-Tate representation with all the weights zero,
and thus it is potentially unramified by Theorem~\ref{thm:Sen}.
Hence if we denote the inertia group of $k'$ by $I_{k'}$,
$\rho(I_{k'})$ is finite.

By construction, the mod $p^2$ representation
\[
 G_{k'}\stackrel{\rho}{\lra}\GL(V)\lra \GL(V/p^2V)
\]
is trivial. On the other hand, $\Ker (\GL(V)\ra \GL(V/p^2V))$ does not contain elements of finite order except the identity.
Thus we see that $\rho(I_{k'})$ is trivial and hence $\rho$ is an unramified representation. 
In particular, $\rho$ is semistable.
\end{proof}

\begin{rem}
As mentioned in Introduction,
it is an interesting question whether one can extend Colmez's strategy
(\cite{Colmez}) to prove the relative $p$-adic monodromy conjecture
using Theorem~\ref{thm:weight zero p-adic monodromy}.
\end{rem}

\end{document}